\documentclass{amsart}
\usepackage[utf8]{inputenc}
\usepackage{adjustbox}
\usepackage{amsmath,amssymb, amsfonts} 

\usepackage[TS1,T1]{fontenc}
\usepackage{listings}
\usepackage{xcolor}

\usepackage{xfrac}

\usepackage[breaklinks,colorlinks,citecolor=blue,urlcolor=blue]{hyperref}

\usepackage[numbers,square,sort,elide]{natbib}
\newtheorem{theorem}{Theorem}[section]
\newtheorem{lemma}{Lemma}[section]
\newtheorem{proposition}{Proposition}[section]
\newtheorem{fact}{Fact}[section]
\newtheorem{corollary}[theorem]{Corollary}
\newtheorem{definition}{Definition}[section]

\newtheorem{conjecture}{Conjecture}[section]
\newcommand{\coloneqq}{\mathrel{\mathord{:}\mathord{=}}}

\usepackage{bm} 
\usepackage{tikz}
\usetikzlibrary{calc} 
\usetikzlibrary{shapes.geometric}
\usepackage{graphicx}
\RequirePackage{amsmath, soul} 
\RequirePackage{amssymb} 
\RequirePackage{latexsym} 
\RequirePackage[mathscr]{eucal} 
\RequirePackage{verbatim} 
\RequirePackage{enumerate} 
\RequirePackage{xspace}
\RequirePackage[colorlinks]{hyperref} 
\RequirePackage[all]{xy}
\RequirePackage{color}
\RequirePackage{graphicx}

\renewcommand{\t}{{\bf t}}

\newcommand{\f}{{\bf f}}

\newcommand\A{{\mathbf A}}
\newcommand\B{{\mathbf B}}

\bmdefine{\boldstar}{\mathchoice{\textstyle*}{\textstyle*}{\textstyle*}{\scriptstyle*}}

\newcommand\Alg[1]{\if#1*\operatorname{\mathsf{Alg}*}\else\operatorname{\mathsf{Alg}}#1\fi}

\newcommand\Mod[1]{\if#1*\operatorname{\mathsf{Mod}*}\else\operatorname{\mathsf{Mod}}#1\fi}

\bmdefine{\Leibniz}{\Omega}        

\bmdefine{\frege}{\Lambda}         

\makeatletter
\newcommand{\tarskidsp}{\mathord%
   {\m@th\raisebox{0pt}[0pt][0pt]{$\stackrel%
   {\raisebox{-2.7pt}[0ex][0pt]{$\displaystyle \,\?\thicksim$}}%
   {\displaystyle\Leibniz}$}}}
\newcommand{\tarskitxt}{\mathord%
   {\m@th\raisebox{0pt}[0pt][0pt]{$\stackrel%
   {\raisebox{-2.7pt}[0ex][0pt]{$\,\?\thicksim$}}{\displaystyle\Leibniz}$}}}
\newcommand{\tarskiscr}{\mathord%
   {{\m@th\raisebox{0pt}[0pt][0pt]{$\stackrel%
   {\raisebox{-2.4pt}[0ex][0pt]{$\scriptstyle \,\?\thicksim$}}%
   {\scriptstyle\Leibniz}$}}}}
\newcommand{\tarskiscrscr}{\mathord%
   {{\m@th\raisebox{0pt}[0pt][0pt]{$\stackrel%
   {\raisebox{-2pt}[0ex][0pt]{$\scriptscriptstyle \,\?\thicksim$}}%
   {\scriptscriptstyle\Leibniz}$}}}}
\newcommand{\Tarski}{\@ifnextchar ^ %
   {\mathchoice{\tarskidsp\kern-.07em}{\tarskitxt\kern-.07em}%
   {\tarskiscr\kern-.07em}{\tarskiscrscr\kern-.07em}}%
   {\mathchoice{\tarskidsp}{\tarskitxt}{\tarskiscr}{\tarskiscrscr}}}
\makeatother
\DeclareMathAlphabet{\mathbfsf}{\encodingdefault}{\sfdefault}{bx}n
\makeatletter
\providecommand*{\Dashv}{\mathrel{\mathpalette\@Dashv\vDash}}
\newcommand*{\@Dashv}[2]{\reflectbox{$\m@th#1#2$}}
\makeatother

\renewcommand\geq{\geqslant}

\newcommand\PL{{\mathbf{P}{\text{\rm \l}}}}
\newcommand\DPL{{\mathbf{DP}{\text{\rm \l}}}}

\newcommand{\alga}{\mathbf A}

\begin{document}
\title[]{Varieties of De Morgan bisemilattices}
\author{F. Paoli, D. Szmuc, A. Borzi, M. Zirattu}
\date{}
\begin{abstract}
    De Morgan bisemilattices are expansions of distributive bisemilattices by an involution satisfying De Morgan properties. They have attracted interest both as algebraic models of analytic containment logics, and as a case study for a certain generalisation of the P\l onka sum construction (\emph{De Morgan-P\l onka sums}). In this paper, we provide a complete description of the lattice of subvarieties of the variety $\mathcal{DMBL}$ of De Morgan bisemilattices. For each subvariety in the lattice, we identify a finite set of finite generators, a characterisation of the De Morgan-P\l onka representations of its members, and a syntactic description of its valid identities. In many cases, we also give an axiomatisation relative to $\mathcal{DMBL}$.


    \emph{MSC Classification 2020: }06A12, 06D30, 08B15.
\end{abstract}
\maketitle

\section{Introduction}

Let $\mathcal{L}$ be a type without constants and with at least an operation symbol of arity $2$ or greater. An $\mathcal{L}$-identity $\varphi \approx \psi$ is \emph{regular} when $\varphi$ and $\psi$ contain precisely the same variables, and a variety $\mathcal{V}$ of type $\mathcal{L}$ is \emph{regular} if it satisfies only regular identities. Although the literature on regular varieties is fairly extensive (see, e.g., \cite{Romanowska, RS}), the only class of regular varieties that is thoroughly understood from the viewpoint of universal algebra consists of the so-called \emph{regularisations of strongly irregular varieties}, which admit satisfactory representations via the classical construction of \emph{P\l onka sums} \cite{Plo67, Plo69, Romanowska, Romanowska92, RS, BonzioPaoliPraBaldi}. The investigation of regular varieties lying outside this class calls for generalisations of the P\l onka construction---generalisations that sometimes also prove effective in the study of certain irregular varieties. These include Lallement sums \cite{RS}, functorial or Agassiz sums \cite{grsic, RS}, involutorial P\l onka sums \cite{Dolvin}, enriched P\l onka sums \cite{fusco1}.

An interesting proposal aimed at extending the scope of the P\l onka construction was recently put forward by Thomas Randriamahazaka under the name of \emph{De Morgan-P\l onka sums} \cite{RandriamahazakaSL}. As in Dolinka and Vin\v{c}ic's involutorial P\l onka sums, this construction employs involutive semilattice direct systems of algebras instead of P\l onka's original semilattice direct systems. However, Dolinka and Vin\v{c}ic's approach is crucially modified so as to be better suited for the investigation of algebras on \emph{dualised} languages, whose operation symbols are paired to exhibit a (generalised form of a) De Morgan duality with respect to a designated unary operation of \emph{negation}. Subject to certain provisos that will be made precise later in the paper, Randriamahazaka proves a generalisation of P\l onka's theorem according to which the identities preserved by his construction are the \emph{balanced regular identities}. These are regular identities $\varphi \approx \psi$ such that a variable occurs in $\varphi$ in the scope of an even (respectively, odd) number of negations if and only if it occurs in $\psi$ in the scope of an even (respectively, odd) number of negations. He also proves that the algebras representable by his method are precisely the \emph{a-involutive} left-normal bands with compatible operations---where, characteristically, an a-involutive band differs from an involutive band because it satisfies the identity $\lnot (x \cdot y) \approx \lnot x \cdot \lnot y$ instead of $\lnot (x \cdot y) \approx \lnot y \cdot \lnot x$.

As an application, Randriamahazaka gives a representation of \emph{De Morgan bisemilattices} (called \textquotedblleft Angellic algebras\textquotedblright ~in his paper), that is, expansions of distributive bisemilattices by an involution satisfying De Morgan properties. In addition to its intrinsic universal algebraic interest, the variety $\mathcal{DMBL}$ of De Morgan bisemilattices---which is regular, but does not regularise any strongly irregular variety---is important in logic as an algebraic semantics for analytic containment logics \cite{angell1989, FineAngell, Fergusonbook}. De Morgan bisemilattices are studied in \cite{PaoliSzmucZirattu}, where results are obtained on varietal and quasivarietal generators of $\mathcal{DMBL}$, on its subdirectly irreducible algebras, and partial results are given concerning the structure of its lattice of subvarieties. A complete description of such a lattice, however, was left open. This paper aims at filling this gap.

The article is structured as follows. Section \ref{prel} contains some required preliminary notions on De Morgan lattices, on bands and semilattices, on P\l onka sums and regular varieties, and on De Morgan-P\l onka sums. Section \ref{invosemi} establishes several auxiliary results on bands and semilattices equipped with an involution, which will later be used in the proof of the main structure theorems. Also, given a variety $\mathcal{V}$ of dualised type, we introduce---in addition to its regularisation $R(\mathcal{V})$ and its balanced regularisation $B(\mathcal{V})$---two other types of regularisations that were not previously considered in the literature: the \emph{bipolarly balanced regularisation} $Bip(\mathcal{V})$ and the \emph{regular bipolarly balanced regularisation} $R(Bip(\mathcal{V}))$. In Section \ref{morgana} we introduce the variety $\mathcal{DMBL}$ of De Morgan bisemilattices and recall Randriamahazaka's theorem stating that $\mathcal{DMBL}$ is the balanced regularisation of the variety $\mathcal{DML}$ of De Morgan lattices. We then study the varieties $B(\mathcal{DML}), R(\mathcal{DML}), Bip(\mathcal{DML})$ and $R(Bip(\mathcal{DML}))$, providing for each an equational basis, a finite set of finite generators, and a description of the De Morgan-P\l onka representations of their members. The same types of regularisations are subsequently examined for the \emph{proper} subvarieties of $\mathcal{DML}$, namely Kleene lattices, Boolean algebras, and the trivial variety. In a concluding subsection, we observe that there exist De Morgan bisemilattices that are neither involutive semilattices nor generators of varieties extending the variety of Boolean algebras. The varieties generated by these algebras do not coincide with any of the regularisations considered above and therefore require separate investigation. In Section \ref{fireworks}, we conclude by presenting a complete description of the lattice of subvarieties of $\mathcal{DMBL}$.\footnote{This article features relabellings of some recently introduced terminology, e.g. in \cite{PaoliSzmucZirattu}. For starters, bipolarly balanced regularisations are therein called conditionally balanced, whereas regular bipolarly balanced ones are called strictly conditionally balanced; similarly, regularly absorptive cases are called semiabsorptive. We think the current nomenclature could, perhaps, be more transparent to the fact that each of these varieties are characterised by their satisfaction of the appropriately described variations of the absorption law.}

For all unexplained notation or terminology on universal algebra, the reader is referred to \cite{BurrisSankappanavar}.


\section{Preliminaries}\label{prel}

\subsection{De Morgan lattices}\label{Belnapsection}

A \emph{De Morgan lattice} \cite{Kalman, Monteiro, Rasiowa} is an algebra $\mathbf{L} = \langle L, \land, \lor, \lnot \rangle$ of type $\langle 2,2,1 \rangle$ such that $\langle L, \land, \lor \rangle$ is a distributive lattice and the identities $\lnot \lnot x \approx x$, $\lnot (x \land y) \approx \lnot x \lor \lnot y$ and $\lnot (x \lor y) \approx \lnot x \land \lnot y$ are satisfied. As a consequence, De Morgan lattices form a variety, hereafter noted $\mathcal{DML}$. The prime example of a De Morgan lattice is the $4$-element algebra $\mathbf{DM}_4$ whose Hasse diagram appears in Figure \ref{fig:DM4}.

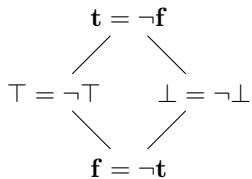
\begin{figure}[ht]
    \centering

\begin{tikzpicture}[scale=.5]
  \node (one) at (0,2) {$\t = \lnot \f$};
  \node (a) at (-2,0) {$\top = \lnot \top$};
  \node (b) at (2,0) {$\bot = \lnot \bot$};
  \node (zero) at (0,-2) {$\f = \lnot \t$};
  \draw (zero) -- (a) -- (one) -- (b) -- (zero);
\end{tikzpicture}

    \caption{The 4-element De Morgan lattice $\mathbf{DM}_4$}
    \label{fig:DM4}
\end{figure}

Observe that the $2$-element Boolean algebra $\mathbf{B}_2$ is isomorphic to the subalgebra of $\mathbf{DM}_4$ with universe $ \{ \t, \f \}$. The only other proper nontrivial subalgebra (up to isomorphism) of $\mathbf{DM}_4$ is the $3$-element algebra $\mathbf{K}_3$, with universe $\{ \t, \f, \top \}$ (or equivalently, $\{ \t, \f, \bot \}$). As a matter of fact, we have that:

\begin{theorem}\cite{Kalman}\label{periscopio}
The following holds:
\begin{enumerate}
    \item The only nontrivial subdirectly irreducible members of $\mathcal{DML}$ are $\mathbf{DM}_4$, $\mathbf{K}_3$ and $\mathbf{B}_2$. 
    \item The unique nontrivial proper subvarieties of $\mathcal{DML}$ are the variety $\mathcal{KL}=V(\mathbf{K}_3)$ of Kleene lattices and the variety $\mathcal{BA}=V(\mathbf{B}_2)$ of Boolean algebras.
    \item  $\mathbf{DM}_4$ generates $\mathcal{DML}$ as a quasivariety.
\end{enumerate}
\end{theorem}

\subsection{Bands and semilattices}

Recall that a \emph{band} \cite[Ch. 2]{Petrich} is an idempotent semigroup $\langle A, \cdot \rangle$. When it's not prejudicial to comprehension, occurrences of $\cdot$ will be replaced by plain juxtaposition. A band is called \emph{left-normal} if it satisfies the identity $xyz \approx xzy$, and it is called \emph{right-normal} if it satisfies the mirrored identity  $xyz \approx yxz$. It is easily verified that a band is commutative iff it is both left-normal and right-normal. In such a case, it is called a \emph{semilattice}, and the product is often denoted by the more evocative symbol $\lor$. A band is called \emph{left-zero} if it satisfies the identity $xy \approx x$, and it is called \emph{right-zero} if it satisfies the mirrored identity  $xy \approx y$. Clearly, any band that is both left-zero and right-zero is trivial. Recall, moreover, that a \emph{double band} \cite{LL} is an algebra $\alga = \langle A, \cdot, + \rangle$ of type $\langle 2,2 \rangle$ such that both $\langle A, \cdot \rangle$ and $ \langle A, + \rangle$ are bands. 

An \emph{i-band} is an algebra $ \langle A,\cdot,\lnot \rangle$ of type $\langle 2,1 \rangle$ such that $\langle A,\cdot \rangle$ is a band and $\lnot$ is a unary operation satisfying $\lnot \lnot x \approx x$. Observe that i-bands always have a double band term reduct, by setting $x + y := \lnot (\lnot x \cdot \lnot y)$. An \emph{involution band} (or \emph{involutive band}) \cite{Dolinka2000} is an i-band that satisfies the identity $\ xy \approx y + x$, while an \emph{a-involutive band} satisfies the identity $\ xy \approx x + y$. While the main motivation for the former identity is generalising the analogous property of groups, the latter is better suited for the study of P\l onka-type decompositions (see below, Section \ref{plonkplink}). Clearly, an involutive band whose underlying band is a semilattice is also an a-involutive band, and is called an \emph{involutive semilattice}. It may be interesting to observe that all left-normal involutive bands are involutive semilattices, while there are non-commutative left-normal a-involutive bands. The variety of involutive semilattices will be denoted by $\mathcal{ISL}$.


Let $\langle A, \cdot \rangle$ be a band. As is well-known \cite[Ch. 2]{Grillet} we can define the following binary relations on $A$, letting for any $a,b \in A$:
\begin{eqnarray*}
    a \leq_L b&:=&ab = a \\
     a \leq_R b&:=&ba = a \\
     a \leq_D b&:=&aba = a \\
     a \leq_H b&:=&ab = a = ba.
\end{eqnarray*}
All these relations are preorders on $A$, and ${\leq_H} = {\leq_L} \cap {\leq_R}$ is a partial order. They are collectively known as \emph{Green's preorders} on the band $\langle A, \cdot \rangle$, and are respectively called the \emph{left, right}, and \emph{natural preorders} and the \emph{natural partial order} on $\langle A, \cdot \rangle$. Observe that in a left-normal band ${\leq_L} = {\leq_D}$, so we only have three Green's preorders to consider. The equivalence relation induced by $\leq_D$ is a congruence in any band. It is generally noted $\mathcal{D}$ and called the \emph{natural Green's relation} of $\langle A, \cdot \rangle$. The following classical result is known as the  Clifford-McLean theorem (see e.g. \cite[Thm. 4.4.1]{Howie}):
\begin{theorem}
    If $\mathbf{A}$ is a band, $\A / \mathcal{D}$ is the maximal semilattice quotient of $\A$.
\end{theorem}

In any left-normal double band, each of the reducts $\langle A, \cdot \rangle$ and $\langle A, + \rangle$ has its associated Green's preorders; we denote them as ${\leq_L^{\cdot}},{\leq_R^{\cdot}}$, $\leq_H^{\cdot}$, ${\leq_L^{+}}, {\leq_R^{+}}$, $\leq_H^{+}$, where the notation is self-explanatory.

\subsection{P\l onka sums and regular varieties}\label{plonkplink}

A  key tool for the investigations that follow is a well-known algebraic construction introduced by Jerzy P\l onka in \cite{Plo67,Plo69} (see also \cite{BonzioPaoliPraBaldi, RS}) and later dubbed \emph{P\l onka sum}. 

Henceforth, let $\mathcal{L}$ be a similarity type without constants, containing at least an operation symbol of arity $2$ or greater, and let $\mathbf{I}=\langle I, \lor \rangle$ be a semilattice. A \emph{semilattice direct system} of $\mathcal{L}$-algebras is an $\mathbf{I}$-shaped diagram in the algebraic category $\mathfrak{A}$ of $\mathcal{L}$-algebras, namely, a covariant functor $F$ from $\mathbf{I}$ (viewed as a small category) to $\mathfrak{A}$. Given $i \in I$, the $\mathcal{L}$-algebra $F(i)$ will be called a \emph{fibre} of the semilattice direct system, while, for $i,j \in I$ such that $i \lor j = j$, the homomorphism $F(\langle i,j \rangle): F(i) \to F(j)$ will be noted $p_{ij}$ and called a \emph{transition map}.


Given a semilattice direct system $F$ of $\mathcal{L}$-algebras with underlying semilattice $\mathbf{I} = \langle I, \vee \rangle$, the \emph{P\l onka sum} over $F$ is the $\mathcal{L}$-algebra $\A= \PL(F) $ whose universe is the disjoint union $A=\displaystyle\bigsqcup_{i\in I} F(i) $ and where a generic $n$-ary term operation $g^\A$ is computed as follows: 
\begin{equation*}\label{eq: operazioni}
g^{\A}(a_{1},\dots, a_n)\coloneqq g^{F(k)}(p_{i_{1}k}(a_1),\dots p_{i_{n}k}(a_n)),
\end{equation*}
where $k= i_{1}\vee\dots\vee i_{n}$ and $a_1\in F(i_1),\dots, a_{n}\in F(i_n)$. The \emph{fibres} of a P\l onka sum are the fibres of its underlying semilattice direct system. A fibre is \emph{trivial} if its universe is a singleton. 

P\l onka sums are especially apt for the study of the so-called regularisations of strongly irregular varieties. Indeed, if $\varphi$ is an $\mathcal{L}$-formula, let $Var(\varphi)$ stand for the set of variables appearing in $\varphi$. An $\mathcal{L}$-identity $\varphi \approx \psi$ is called \textit{regular} if $Var(\varphi) = Var(\psi)$, and \emph{irregular} otherwise. A variety of $\mathcal{L}$-algebras is \emph{strongly irregular} if it satisfies an irregular identity of the form $x \approx \varphi(x,y)$, where the variables $x,y$ essentially occur in $\varphi(x,y)$. Given a strongly irregular variety $\mathcal{V}$, its \textit{regularisation} $R(\mathcal{V})$ is the variety that satisfies all and only the regular identities valid in $\mathcal{V}$. Among the results proved by P\l onka, the following two are especially remarkable:

\begin{theorem}\cite[Thm. I]{Plo67}\label{affarone}
If $F$ is a semilattice direct system of $\mathcal{L}$-algebras with underlying semilattice $\mathbf{I}$, and the $2$-element semilattice $\mathbf{S}_2$ is a subalgebra of $\mathbf{I}$, all regular identities satisfied in all fibres of $F$ are satisfied in $\PL(F)$, whereas any other identity is not satisfied in $\PL(F)$.
\end{theorem}

For the next result, let us define a \emph{left-normal band with compatible operations} as a $\mathcal{L}$-algebra $\A$ for which there is a binary operation $\cdot$ on $A$ (not necessarily a term operation of $\A$) such that $\langle A, \cdot \rangle$ is a left-normal band and for all $n$-ary $g \in \mathcal{L}$ and every $a,b_1,...,b_n \in A$, the conditions $a \cdot g^\A (b_1,...,b_n) = ab_1,...,b_n$ and $g^\A (b_1,...,b_n) \cdot a = g^\A (b_1a,...,b_na)$ are satisfied.

\begin{theorem}\cite[Thm. II]{Plo67}; \cite[Thm. 7.1]{Romanowska92}\label{cappelletti}
    Let $\A$ be an $\mathcal{L}$-algebra. $\A$ is a left-normal band with compatible operations iff there exists a semilattice direct system $F$ of $\mathcal{L}$-algebras such that $\A = \PL (F)$. Moreover, if $\mathcal{V}$ is a strongly irregular variety of $\mathcal{L}$-algebras, $\A \in R(\mathcal{V})$ iff $\A = \PL (F)$, for some semilattice direct system $F$ whose fibres belong to $\mathcal{V}$.
\end{theorem}

Interestingly, the left-to-right implication in the first statement of Theorem \ref{cappelletti} is proved by constructing a semilattice direct system over the algebra $\langle A, \cdot \rangle / \mathcal{D}$ whose fibres are the $\mathcal{D}$-equivalence classes, which are left-zero bands with compatible operations.

Of special interest, for the considerations that follow, is of course the regularisation of the variety $\mathcal{DML}$ of De Morgan lattices, studied and axiomatised in \cite{Hornischer, PaoliSzmucZirattu}. This is one of the proper subvarieties of De Morgan bisemilattices about which the present paper is focussing. Its members will be called \emph{regularised absorptive} De Morgan bisemilattices (in \cite{PaoliSzmucZirattu}, they were called \emph{semiabsorptive}) because they satisfy the restricted absorption identity $x \land (x \lor y) \approx x \land (x \lor \lnot y)$, on which see Theorem \ref{servosterzo} below.

\subsection{De Morgan-P\l onka sums}\label{inquina}

In a semilattice direct system, all unary operations map their arguments to elements in the same fibre. There have been some proposals to lift this overly restrictive condition, among which we mention Dolinka and Vin\v{c}i\'c's \emph{involutorial P\l onka sums} \cite{Dolvin}. Here, we focus on a different suggestion, recently advanced by T. Randriamahazaka in an attempt to represent algebras with a De Morgan negation that can be built via a certain form of a twist product construction out of algebras with a dualised type \cite{RandriamahazakaSL}.

Let $\mathcal{L}$ be a similarity type, and let $d$ be a function that assigns to each $n$-ary symbol in $\mathcal{L}$ an $n$-ary symbol in $\mathcal{L}$. $\mathcal{L}$ is said to be \textit{dualised} w.r.t. $d$ if, for all $f^n \in \mathcal{L}$, $d(d(f^n)) = f^n$. If $\mathcal{L}$ is a dualised type w.r.t. $d$, every $\mathcal{L}$-algebra $\mathbf{A} = \langle A, \lbrace f^{\mathbf{A}}\rbrace_{f \in \mathcal{L}} \rangle$ induces an $\mathcal{L}$-algebra $\mathbf{A}^d = \langle A, \lbrace d(f)^{\mathbf{A}} \rbrace_{f \in \mathcal{L}} \rangle$. A class of $\mathcal{L}$-algebras is called \emph{symmetric} if it contains $\mathbf{A}^d$ whenever it contains $\mathbf{A}$.

For the definition that follows, recall that any involutive semilattice $\langle I, \lor, \lnot \rangle$ can be viewed as an internal category in the category of $\mathbf{Z}_2$-sets, obtained by enriching the small category $\langle I, \lor \rangle$ by an action of $\mathbf{Z}_2$ on $I$ whose behaviour is determined by $\lnot$. Its orbit isomorphisms are generated by $a \simeq \lnot a$; the condition that $a \leq b$ implies $\lnot a \leq \lnot b$ guarantees equivariance under such isomorphisms.

Let $\mathcal{L}$ be a dualised type (w.r.t. $d$). 
An \textit{involutive semilattice direct system} of $\mathcal{L}$-algebras is a $\mathbf{Z}_2$-equivariant preorder functor $F$ from $\langle I, \lor, \lnot \rangle$, viewed as an internal category in the category of $\mathbf{Z}_2$-sets, to the algebraic category of $\mathcal{L}$-algebras, subjected to the condition that $F(\lnot i) = F(i)^d$. As for semilattice direct systems of $\mathcal{L}$-algebras, the $\mathcal{L}$-algebra $F(i)$ will be called a \emph{fibre} of the system, while, for $i,j \in I$ such that $i \leq_{\mathbf{I}}  j $, the homomorphism $F(\langle i,j \rangle): F(i) \to F(j)$ will be noted $p_{ij}$ and called a \emph{transition map}. Moreover, the involution isomorphisms from $F(i)$ to $F(\lnot i)$ will be noted $n_{i,\lnot i}$ and called \emph{dualising maps}.

Observe that our definition implies $n_{\lnot i,i} = n_{i,\lnot i}^{-1}$, and that $\mathbf{Z}_2$-equivariance implies that for all $i,j \in I$, $n_{j, \lnot j} \circ p_{ij} = p_{\lnot i,\lnot j} \circ n_{i,\lnot i}$ whenever $i \leq_{\mathbf{I}}  j $. Also, observe that semilattice direct systems can be identified with those involutive semilattice direct systems whose underlying involutive semilattice satisfies $\lnot x \approx x$.

If $F$ is an involutive semilattice direct system of $\mathcal{L}$-algebras, the \textit{De Morgan-P\l onka sum} over $F$ is the algebra $\DPL(F)$, of type $\mathcal{L} \cup  \{ \lnot\}$, whose universe is the disjoint union $A=\displaystyle\bigsqcup_{i\in I} F(i) $ and such that: 
\begin{itemize}
\item for every $n$-ary $f \in \mathcal{L}$ (with $n\geq 1$) and any $a_{1},\dots,a_{n}\in A$,
\[
f^{\DPL(F)}(a_{1},\dots,a_{n})=f^{F(i)}(p_{i_{1}i}(a_{1}),\dots,p_{i_{n}i}(a_{n})),
\]
where $a_{1}\in F(i_{1}),\dots,a_{n}\in F(i_{n})$ and $i=i_{1}\vee\dots\vee i_{n}$;
\item[] 
\item $\lnot^{\DPL(F)} a = n_{i,\lnot i}(a)$, where $a\in F(i)$.
\end{itemize}



In the interests of brevity, if $\mathcal{L}$ is as above, we denote by $\mathcal{L}^\ast$ the expanded type $\mathcal{L} \cup  \{ \lnot\}$. If $\mathcal{V}$ is an $\mathcal{L}$-variety, $\mathcal{V}^\ast$ will denote the $\mathcal{L}^\ast$-variety whose equational theory is obtained by closing the equational theory of $\mathcal{V}$ under the identities $\lnot \lnot x \approx x$ and $\lnot f(x_1,...,x_n) \approx d(f) (\lnot x_1, ..., \lnot x_n)$, for all $n$-ary $f \in \mathcal{L}$. Observe that if $F$ is an involutive semilattice direct system of $\mathcal{L}$-algebras with underlying involutive semilattice $\mathbf{I}$, and $i \in I$ is such that $\lnot^\mathbf{I} i = i$, the $\mathcal{L}^\ast$-algebra $F(i)^\lnot$ obtained by expanding $F(i)$ with the negation $\lnot x := n_{ii}(x)$ is a subalgebra of $\DPL(F)$ and belongs to $\mathcal{V}^\ast$.

If $\varphi$ is an $\mathcal{L}^\ast$-formula, let $Var^{+}(\varphi)$ and $Var^{-}(\varphi)$ stand for the set of variables that occur in $\varphi$ in the scope of an even, respectively odd, number of occurrences of $\lnot$. An $\mathcal{L}^\ast$-identity $\varphi \approx \psi$ is said to be \textit{balanced  regular} if $Var^{+}(\varphi) = Var^{+}(\psi)$ and $Var^{-}(\varphi) = Var^{-}(\psi)$. Given a strongly irregular variety $\mathcal{V}$ of type $\mathcal{L}^\ast$, its \textit{balanced regularisation} $B(\mathcal{V})$ is the variety that satisfies all and only the balanced regular identities valid in $\mathcal{V}$.

If $\mathbf{A} = \langle A, \lbrace f^{\mathbf{A}} \rbrace_{f \in \mathcal{L}} \rangle$ is an $\mathcal{L}$-algebra, the \textit{bilateralisation} of $\mathbf{A}$ is the algebra $\flat\mathbf{A}$, of type $\mathcal{L}^\ast$, with universe $A^2$ and such that:
\begin{itemize}
    \item $f^{\flat\mathbf{A}} (\langle a_1,b_1 \rangle,\dots,\langle a_n,b_n\rangle) = \langle f^{\mathbf{A}} (a_1,\dots,a_n), d(f)^{\mathbf{A}} (b_1,\dots,b_n)\rangle$, for every $n$-ary $f \in \mathcal{L}$ (with $n\geq 1$) and any $a_{1},\dots,a_{n}, b_{1},\dots,b_{n}\in A$;
    \item[]
    \item $\lnot^{\flat\mathbf{A}}\!\langle a,b \rangle = \langle b, a \rangle $.
\end{itemize}

It turns out that De Morgan-P\l onka sums preserve the balanced regular identities satisfied in all algebras that can be obtained from the fibres via bilateralisation (which can be seen as a general version of the twist product construction, see e.g. \cite{Moraschini}). This yields the following generalisation of Theorem \ref{affarone}:

\begin{theorem}
 \cite{RandriamahazakaSL} \label{th:DM-Plonka sums and balanced equations} If $\mathcal{L}$ is a dualised type and $F$ is an involutive semilattice direct system of $\mathcal{L}$-algebras with underlying involutive semilattice $\mathbf{I}$, then all balanced regular identities satisfied in the bilateralisations of all fibres of $F$ are satisfied in $\DPL(F)$, whereas, if the $4$-element involutive semilattice ${\bf IS_4}$ is a subalgebra of $\mathbf{I}$, identities that fail to be balanced regular are not satisfied in $\DPL(F)$.   
\end{theorem} 

The paper \cite{RandriamahazakaSL} also contains a generalisation of Theorem \ref{cappelletti}. Let $\mathcal{L}$ be a dualised type. Let us define an \emph{a-involutive left-normal band with compatible operations} as an algebra $\A$, of type $\mathcal{L}^\ast$, such that the $\lnot$-free reduct of $\A$ is a left-normal band with compatible operations and with underlying band $\langle A, \cdot \rangle$, and such that $\langle A, \cdot, \lnot \rangle$ is an a-involutive band.

\begin{theorem}\label{cappelloni}
    Let $\A$ be a $\mathcal{L}^\ast$-algebra. $\A$ is an a-involutive left-normal band with compatible operations iff there exists an involutive semilattice direct system $F$ of $\mathcal{L}$-algebras such that $\A = \DPL (F)$. Moreover, if $\mathcal{V}$ is a strongly irregular symmetric variety of $\mathcal{L}$-algebras, $\A \in B(\mathcal{V}^\ast)$ iff $\A = \DPL (F)$, for some involutive semilattice direct system $F$ whose fibres belong to $\mathcal{V}$.
\end{theorem}

The left-to-right implication in the first statement of Theorem \ref{cappelloni} is proved by showing that, under the given assumption, the Green's relation $\mathcal{D}$ is an a-involutive band congruence, so that it is possible to construct an involutive semilattice direct system over the involutive semilattice quotient $\langle A, \cdot, \lnot \rangle / \mathcal{D}$.

An especially interesting case of Theorem \ref{cappelloni} arises when $\mathcal{V}$ is the variety $\mathcal{DL}$ of distributive lattices, and consequently, $\mathcal{V}^\ast$ is the variety $\mathcal{DML}$ of De Morgan lattices. According to Theorem \ref{cappelloni}, $B(\mathcal{DML})$ consists precisely of the De Morgan-P\l onka sums of distributive lattices (see Theorem \ref{accozzato} below).


\section{I-bands and involutive semilattices}\label{invosemi}

In this section, we establish certain facts about i-bands and involutive semilattices that will be useful in our study of varieties of De Morgan bisemilattices. For a start, we characterise the identities of type $\langle 2, 1 \rangle$ that are satisfied by each variety of involutive semilattices.

In \cite{crvdol} it was proved that there are four subdirectly irreducible members of $\mathcal{ISL}$, which we call ${\bf IS_1}$, ${\bf IS_2}$, ${\bf IS_3}$, ${\bf IS_4}$, displaying them in Figure \ref{involtini}. 

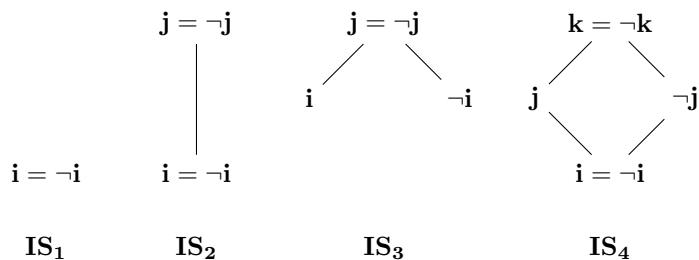
\begin{figure}[ht]
\begin{tikzpicture}[scale=.5]
    \begin{scope}[shift={(0,0)}] 
  \node (zero) at (0,-2) {${\bf i} = \lnot {\bf i}$};
  \node (tag) at (0,-4) {${\bf IS_1}$};  
    \end{scope}

    \begin{scope}[shift={(4,0)}] 
  \node (one) at (0,2) {${\bf j} = \lnot {\bf j}$};
  \node (zero) at (0,-2) {${\bf i} = \lnot {\bf i}$};
  \node (tag) at (0,-4) {${\bf IS_2}$};
  \draw (one) -- (zero);
    \end{scope}

    \begin{scope}[shift={(9,0)}] 
  \node (one) at (0,2) {${\bf j} = \lnot {\bf j}$};
  \node (a) at (-2,0) {${\bf i}$};
  \node (b) at (2,0) {$\lnot {\bf i}$};
  \node (tag) at (0,-4) {${\bf IS_3}$};
  \draw (a) -- (one) -- (b) ;
    \end{scope}

    \begin{scope}[shift={(15,0)}] 
  \node (one) at (0,2) {${\bf k} = \lnot {\bf k}$};
  \node (a) at (-2,0) {${\bf j}$};
  \node (b) at (2,0) {$\lnot {\bf j}$};
  \node (zero) at (0,-2) {${\bf i} = \lnot {\bf i}$};
  \node (tag) at (0,-4) {${\bf IS_4}$};
  \draw (zero) -- (a) -- (one) -- (b) -- (zero);
    \end{scope}

\end{tikzpicture}
\caption{The subdirectly irreducible involutive semilattices}
\label{involtini}
\end{figure}
Altogether, there are precisely five subvarieties of $\mathcal{ISL}$, depicted in Figure~\ref{fig:subvarieties-ISL} together with an axiomatisation of each of them relative to $\mathcal{ISL}$  \cite{Dolinka2000}. $\mathcal{RISL}$ is the variety of \textit{regular} involutive semilattices, generated by ${\bf IS_2}$. $\mathcal{BISL}$ is the variety of \textit{bipolar} involutive semilattices, generated by ${\bf IS_3}$, and $\mathcal{RBISL}$ is the variety of \textit{regular bipolar} involutive semilattices, generated by the class $ \{ \mathbf{IS}_2, \mathbf{IS}_3 \}$. The whole variety $\mathcal{ISL}$ is generated by ${\bf IS_4}$. Finally, $\mathcal{T}$ is the trivial variety of type $\langle 2, 1 \rangle$, generated by ${\bf IS_1}$.

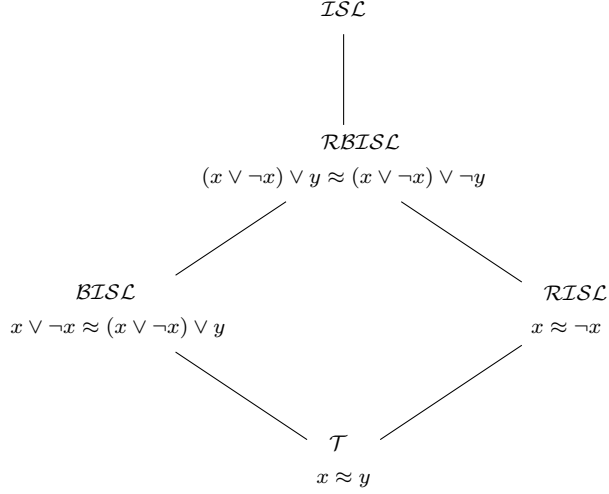
\begin{figure}[ht]
    \centering
\begin{tikzpicture}[scale=.5]
  \node (new) at (0,8) {$\footnotesize 
  \begin{aligned}
   & \mathcal{ISL}\\
  \end{aligned}
  $};
  \node (one) at (0,4)  {$\footnotesize 
  \begin{aligned}
   & \mathcal{RBISL}\\
       (x \lor \lnot x) \lor y  & \approx (x \lor \lnot x) \lor \lnot y\\
  \end{aligned}
  $};
  \node (a) at (-6,0)  {$\footnotesize 
  \begin{aligned}
   & \mathcal{BISL}\\
       x \lor \lnot x  & \approx (x \lor \lnot x) \lor y\\
  \end{aligned}
  $};
  \node (b) at (6,0)   {$\footnotesize 
  \begin{aligned}
   & \mathcal{RISL}\\
       x & \approx \lnot x\\
  \end{aligned}
  $};
  \node (zero) at (0,-4)   {$\footnotesize 
  \begin{aligned}
   & \mathcal{T}\\
       x & \approx y\\
  \end{aligned}
  $};
  \draw (zero) -- (a) -- (one) -- (b) -- (zero);
  \draw (new) -- (one);    
  
\end{tikzpicture}
    
\caption{The lattice of subvarieties of $\mathcal{ISL}$ and their relative axiomatisations}
    \label{fig:subvarieties-ISL}
\end{figure}

We now proceed to characterising the identities satisfied in each one of these subvarieties. Recall from Section \ref{plonkplink} that an $\mathcal{L}$-identity $\varphi \approx \psi$ is said to be regular iff $Var(\varphi)= Var(\psi)$, and from Section \ref{inquina} that, if $\mathcal{L}$ is a dualised type, an $\mathcal{L}^\ast$-identity $\varphi \approx \psi$ is said to be balanced regular iff $Var^+(\varphi)= Var^+(\psi)$ and $Var^-(\varphi)= Var^-(\psi)$.  To begin with, observe that $\mathcal{RISL}$ is term-equivalent to the variety of semilattices---namely, to the regularisation of the trivial variety of type $\langle 2, 1 \rangle$. By Theorem \ref{cappelletti}, $\mathcal{RISL}$ is the class of P\l onka sums of trivial algebras of type $\langle 2, 1 \rangle$. In sum:

\begin{fact}
\label{fact:characterization-IS2}The following conditions are equivalent for any identity $\varphi \approx \psi$ of type $\langle 2, 1 \rangle$:
\begin{enumerate}
    \item ${\bf IS_2} \vDash \varphi \approx \psi$;
    \item $\mathcal{RISL} \vDash \varphi \approx \psi$;
    \item $\varphi \approx \psi$ is regular.
\end{enumerate}    
\end{fact}

Independent proofs of this result were also given in \cite[Thm. 2.3.2]{BonzioPaoliPraBaldi}, \cite{CreDolinkaVinc2000} and \cite{Dolinka2000}. The next fact is proved in \cite[Prop. 2]{RandriAJL}, although it can be derived from Theorem \ref{th:DM-Plonka sums and balanced equations} by observing that $\mathbf{IS}_4$ can be represented as a De Morgan-P\l onka sum of trivial algebras of type $\langle 2 \rangle$.


\begin{fact}
\label{fact:characterization-IS4}The following conditions are equivalent for any identity $\varphi \approx \psi$ of type $\langle 2, 1 \rangle$:
\begin{enumerate}
    \item ${\bf IS_4} \vDash \varphi \approx \psi$;
    \item $\mathcal{ISL} \vDash \varphi \approx \psi$;
    \item $\varphi \approx \psi$ is balanced regular.
\end{enumerate}     
\end{fact}

For the next results, we need to introduce new concepts concerning the syntactic structure of identities in dualised types. If $\mathcal{L}$ is such a type, an $\mathcal{L}$-identity $\varphi \approx \psi$ is said to be \emph{bipolar} if $Var^+(\varphi) \cap Var^-(\varphi) \neq \emptyset$ and $Var^+(\psi) \cap Var^-(\psi) \neq \emptyset$. $\varphi \approx \psi$ is \emph{bipolarly balanced} if it is either bipolar or balanced regular, and is \emph{regular bipolarly balanced} if it is either bipolar and regular, or balanced regular.

\begin{proposition}
\label{prop:characterization-IS3}
The following conditions are equivalent for any identity $\varphi \approx \psi$ of type $\langle 2, 1 \rangle$:
\begin{enumerate}
    \item ${\bf IS_3} \vDash \varphi \approx \psi$;
    \item $\mathcal{BISL} \vDash \varphi \approx \psi$;
    \item $\varphi \approx \psi$ is bipolarly balanced.
\end{enumerate}      
\end{proposition}

\begin{proof}
In light of the above, it will suffice to prove the equivalence of (1) and (3).

\underline{(3) implies (1).} Observe that for any formula $\delta$ and any $h \in Hom(\mathbf{Fm}(\langle 2, 1 \rangle), {\bf IS_3})$, the following can be proved by induction on the complexity of $\delta$:

\[
h(\delta) = \left (\bigvee\limits_{x \in Var^{+}(\delta)} h(x) \right ) \lor \left (\bigvee\limits_{y \in Var^{-}(\delta)} \lnot h(y) \right )
\]

\medskip

\noindent It follows that if $\varphi \approx \psi$ is bipolar, then for all $h \in Hom(\mathbf{Fm}(\langle 2, 1 \rangle), {\bf IS_3})$, $h(\varphi) = {\bf j} = h(\psi)$, whereas if $\varphi \approx \psi$ is balanced regular then for all such $h$, $h(\varphi) = h(\psi)$. Whence, ${\bf IS_3} \vDash \varphi \approx \psi$.

\strut

\underline{(1) implies (3).}  Assume that $\varphi \approx \psi$ is neither bipolar nor balanced. Suppose without loss of generality that $\varphi$ is not bipolar. By the fact that the identity is not balanced, we know that either
(i) $Var^{+}(\psi) \setminus Var^{+}(\varphi) \neq \emptyset$, or
(ii) $Var^{-}(\psi) \setminus Var^{-}(\varphi) \neq \emptyset$, or
(iii) $Var^{+}(\varphi) \setminus Var^{+}(\psi) \neq \emptyset$, or
(iv) $Var^{-}(\varphi) \setminus Var^{-}(\psi) \neq \emptyset$.
In cases (i) and (ii), consider $Hom(\mathbf{Fm}(\langle 2, 1 \rangle), {\bf IS_3})$ such that:

\[
h(x) =
\begin{cases}
{\bf i} & \text{if } x \in Var^{+}(\varphi)\\
{\bf \lnot i} & \text{if } x \in Var^{-}(\varphi)\\
{\bf j} & \text{otherwise } \\
\end{cases}
\]

\medskip

\noindent Note that $h$ is well-defined because $\varphi$ is not bipolar. By induction on the complexity of the formula, we can corroborate that $h(\varphi) = {\bf i}$ while $h(\psi) = {\bf j}$. Cases (iii) and (iv) are analogous, replacing $\varphi$ with $\psi$. The case in which $\psi$ is not bipolar is treated similarly. Whence, ${\bf IS_3} \nvDash \varphi \approx \psi$.
\end{proof}

\begin{corollary}\label{stritto}
The following conditions are equivalent for any identity $\varphi \approx \psi$ of type $\langle 2, 1 \rangle$:
\begin{enumerate}
    \item $\mathbf{IS_2} \times \mathbf{IS_3}  \vDash \varphi \approx \psi$;
    \item $\mathcal{RBISL} \vDash \varphi \approx \psi$;
    \item $\varphi \approx \psi$ is regular bipolarly balanced.
\end{enumerate}       
\end{corollary}

\begin{proof}
Again, the equivalence of (1) and (2) is clear from the definitions and the above remarks.
If $\varphi \approx \psi$ is not regular bipolarly balanced, then it is either irregular (and then it fails in $\mathbf{IS_2}$ by Fact \ref{fact:characterization-IS2}) or not bipolarly balanced (and then it fails in $\mathbf{IS_3}$ by Proposition \ref{prop:characterization-IS3}). So (1) implies (3). If $\varphi \approx \psi$ has a counterexample in $\mathbf{IS_2} \times \mathbf{IS_3} $, then it has a counterexample either in $\mathbf{IS_2} $ or in $\mathbf{IS_3} $, and we use the same results to obtain the implication from (3) to (1).   
\end{proof}

If an algebra $\A$ admits a De Morgan-P\l onka representation over a certain involutive semilattice $\mathbf{I}$, then $\mathbf{I}$ will belong to one or more of the subvarieties of $\mathcal{ISL}$ we have just characterised. On the other hand, by Theorem \ref{cappelloni}, $\A$ is an a-involutive left-normal band with compatible operations such that $\mathbf{I}$ is obtained as the quotient (modulo $\mathcal{D}$) of its a-involutive band reduct. It will be useful, in the sequel, to have at our disposal necessary and sufficient conditions for the membership of this quotient in each of the subvarieties of $\mathcal{ISL}$.

\begin{lemma}
    Let $\mathbf{A}$ be an a-involutive left-normal band. The following conditions are equivalent:
    \begin{enumerate}
        \item $\A \vDash x \cdot y \approx x \cdot \lnot y $;
        \item $\A \vDash x \cdot \lnot x \approx x $;
        \item $\A / \mathcal{D} \in \mathcal{RISL}$.
    \end{enumerate}
\end{lemma}

\begin{proof}
    (1) implies (2). If (1) holds, then for any $a \in A$ we have that $a = aa = a \cdot \lnot a$.

    (2) implies (3). If (2) holds, for any $a$ we have that $a \cdot \lnot a = a$ and $\lnot a \cdot a = \lnot a \cdot \lnot \lnot a = \lnot a$. Hence $[a]_\mathcal{D} = [\lnot a]_\mathcal{D}$, i.e., $\A / \mathcal{D} \in \mathcal{RISL}$.

    (3) implies (1). Let $a,b \in A$. By (3) and the a-involution property, $a \cdot \lnot b = (a \cdot \lnot b) \cdot \lnot (a \cdot \lnot b) = a \cdot \lnot b \cdot \lnot a \cdot b$. But then, using left-normality and (3) again, $a \cdot \lnot b = a \cdot \lnot b \cdot \lnot a \cdot b = a \cdot b \cdot \lnot a \cdot \lnot b = (a \cdot b) \cdot \lnot (a \cdot b) = a \cdot b$.
\end{proof}

\begin{lemma}\label{mezzitermini}
    Let $\mathbf{A}$ be an a-involutive left-normal band. The following conditions are equivalent:
    \begin{enumerate}
        \item $\A \vDash x \cdot \lnot x \approx x \cdot \lnot x \cdot y \cdot \lnot y $;
        \item $\A \vDash x \cdot \lnot x \approx x \cdot \lnot x \cdot y  $;
        \item $\A / \mathcal{D} \in \mathcal{BISL}$.
    \end{enumerate}
\end{lemma}

\begin{proof}
    (1) implies (2). If (1) holds, then given $a,b \in A$, we have that $a \cdot \lnot a \cdot b= a \cdot \lnot a \cdot b \cdot \lnot b \cdot b = a \cdot \lnot a \cdot b \cdot \lnot b = a \cdot \lnot a$.

    (2) implies (3) because identities are preserved by quotients.

    (3) implies (1). If (3) holds, then for any $a,b$ we have that $[a \cdot \lnot a]_\mathcal{D} = [a]_\mathcal{D} \lor \lnot [a]_\mathcal{D} = [a]_\mathcal{D} \lor \lnot [a]_\mathcal{D} \lor [b]_\mathcal{D} \lor \lnot [b]_\mathcal{D}  = [a \cdot \lnot a \cdot b \cdot \lnot b]_\mathcal{D}$. In particular, $a \cdot \lnot a \cdot b\cdot \lnot b = a \cdot \lnot a \cdot a \cdot \lnot a \cdot b \cdot \lnot b = a \cdot \lnot a$.
\end{proof}

\begin{lemma}\label{tiramolla}
    Let $\mathbf{A}$ be an a-involutive left-normal band. The following conditions are equivalent:
    \begin{enumerate}
        \item $\A \vDash x \cdot \lnot x \cdot y \approx x \cdot \lnot x \cdot \lnot y $;
        \item $\A / \mathcal{D} \in \mathcal{RBISL}$.
    \end{enumerate}
\end{lemma}

\begin{proof}
    For the nontrivial implication, observe that if (3) holds, then for any $a,b \in A$ both $a \cdot \lnot a \cdot b$ and $a \cdot \lnot a \cdot \lnot b$ are equal to $a \cdot \lnot a \cdot b \cdot \lnot b$.
\end{proof}

In what follows, given a strongly irregular variety $\mathcal{V}$ of dualised type $\mathcal{L}$, its \textit{bipolarly balanced regularisation} $Bip(\mathcal{V})$ is the variety that satisfies all and only the bipolarly balanced identities valid in $\mathcal{V}$, and its \textit{regular bipolarly balanced regularisation} $R(Bip(\mathcal{V}))$ is the variety that satisfies all and only the regular bipolarly balanced identities valid in $\mathcal{V}$. 
It follows from Facts \ref{fact:characterization-IS2} and \ref{fact:characterization-IS4}, as well as from Proposition \ref{prop:characterization-IS3} and Corollary \ref{stritto}, that if $\mathcal{T}$ is the trivial variety of type $\langle 2, 1 \rangle$, $R(\mathcal{T}) = \mathcal{RISL}, Bip(\mathcal{T}) = \mathcal{BISL}, R(Bip(\mathcal{T})) = \mathcal{RBISL}$, and $B(\mathcal{T}) = \mathcal{ISL}$.

\section{De Morgan bisemilattices}\label{morgana}

Recall that a \emph{distributive bisemilattice} is a double band $\mathbf{B} = \langle B, \land, \lor \rangle$ such that the reducts $\langle B, \land \rangle$ and $\langle B, \lor \rangle$ are semilattices, and $\land$ ($\lor$, respectively) distributes over $\lor$ ($\land$, respectively). Thus, distributive bisemilattices are a variety, noted $\mathcal{DBL}$. In his classic paper \cite{Plo}, P\l onka showed that $\mathcal{DBL} = R(\mathcal{DL})$, where $\mathcal{DL}$ is the variety of distributive lattices.

In the present paper, we are interested in the variety $\mathcal{DMBL} $ of \emph{De Morgan bisemilattices}. A De Morgan bisemilattice is an algebra $\mathbf{B} = \langle B, \land, \lor, \lnot \rangle$ of type $\langle 2, 2, 1 \rangle$ such that $\langle B, \land, \lor \rangle$ is a distributive bisemilattice and the identities $\lnot \lnot x \approx x$, $\lnot (x \land y) \approx \lnot x \lor \lnot y$, and $\lnot (x \lor y) \approx \lnot x \land \lnot y$ are satisfied. De Morgan bisemilattices are called \emph{Angellic algebras} in \cite{RandriamahazakaSL} and studied under a different name also in \cite{Hornischer}. They were extensively investigated in \cite{PaoliSzmucZirattu}. As recalled in the introduction, they are the balanced regularisation of the variety of De Morgan lattices, and hence, the class of all algebras that are representable as De Morgan-P\l onka sums of distributive lattices.

De Morgan lattices are the De Morgan bisemilattices that satisfy $x \approx x \land (x \lor y)$. With a slight abuse, involutive semilattices can be identified with De Morgan bisemilattices that satisfy $x \land y \approx x \lor y$. When referring to an involutive semilattice, the context will suffice to make it clear whether we view it as an algebra of type $\langle 2, 1 \rangle$ or of type $\langle 2, 2, 1 \rangle$. We will keep the same symbols to denote an involutive semilattice in the $\langle 2, 2, 1 \rangle$ type and its $\langle 2, 1 \rangle$-reduct, relying again on the context to disambiguate.

Occasionally, if $F$ is an involutive semilattice direct system of distributive lattices over $\mathbf{IS}_2$ such that $F(\mathbf{j})$ is a trivial algebra and $\A = F(\mathbf{i})^\lnot$, we denote the De Morgan bisemilattice $\DPL(F)$ by $\A^\dagger$.

We are especially interested in investigating, for each subvariety $\mathcal{V}$ of the variety $\mathcal{DML}$ of De Morgan lattices, the varieties $R(\mathcal{V}), Bip(\mathcal{V}), R(Bip(\mathcal{V}))$ and $B(\mathcal{V})$. Clearly, all these varieties are subvarieties of $\mathcal{DMBL}$ (see Figure~\ref{fig:lattice-regularisations}, where the varieties of involutive semilattices $R(\mathcal{T}), Bip(\mathcal{T}), R(Bip(\mathcal{T}))$ and $B(\mathcal{T})$ are not depicted).

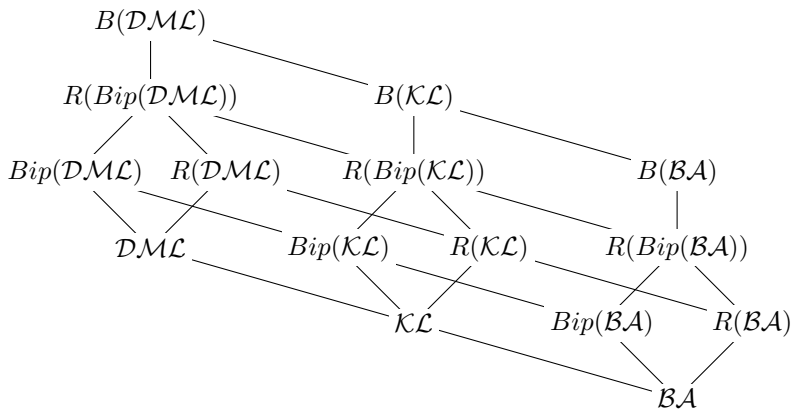
\begin{figure}[ht]
    \centering
    \begin{tikzpicture}[scale=1] 

        \coordinate (V01) at (0,0);
        \coordinate (V02) at (1, 1);
        \coordinate (V03) at (0, 2);
        \coordinate (V0T) at (0, 3);        
        \coordinate (V04) at (-1, 1);

        \coordinate (V11) at (3.5, -1);
        \coordinate (V12) at (4.5, 0);
        \coordinate (V13) at (3.5, 1);
        \coordinate (V1T) at (3.5, 2);        
        \coordinate (V14) at (2.5, 0);

        \coordinate (V21) at (7, -2);
        \coordinate (V22) at (8, -1);
        \coordinate (V23) at (7, 0);
        \coordinate (V2T) at (7, 1);        
        \coordinate (V24) at (6, -1);

        \coordinate (V31) at (10.5, -3);
        \coordinate (V32) at (11.5, -2);
        \coordinate (V33) at (10.5, -1);
        \coordinate (V34) at (9.5, -2);

        \draw (V01) -- 
        (V11);
        \draw (V11) -- 
        (V21);

        \draw (V02) -- 
        (V12);
        \draw (V12) -- 
        (V22);

        \draw (V03) -- 
        (V13);
        \draw (V13) -- 
        (V23);

        \draw (V04) -- 
        (V14);
        \draw (V14) -- 
        (V24);

        \draw (V0T) -- 
        (V1T);
        \draw (V1T) -- 
        (V2T);

        \draw (V01) -- (V02) -- (V03) -- (V04) -- cycle;
        \draw (V03) -- (V0T);
        \node[fill=white, inner sep=2pt] at (V01) {$\mathcal{DML}$};
        \node[fill=white, inner sep=2pt] at (V02) {$R(\mathcal{DML})$};
        \node[fill=white, inner sep=2pt] at (V03) {$R(Bip(\mathcal{DML}))$};
        \node[fill=white, inner sep=2pt] at (V04) {$Bip(\mathcal{DML})$};
        \node[fill=white, inner sep=2pt] at (V0T) {$B(\mathcal{DML})$};        

        \draw (V11) -- (V12) -- (V13) -- (V14) -- cycle;
        \draw (V13) -- (V1T);        
        \node[fill=white, inner sep=2pt] at (V11) {$\mathcal{KL}$};
        \node[fill=white, inner sep=2pt] at (V12) {$R(\mathcal{KL})$};
        \node[fill=white, inner sep=2pt] at (V13) {$R(Bip(\mathcal{KL}))$};
        \node[fill=white, inner sep=2pt] at (V14) {$Bip(\mathcal{KL})$};
        \node[fill=white, inner sep=2pt] at (V1T) {$B(\mathcal{KL})$};      

        \draw (V21) -- (V22) -- (V23) -- (V24) -- cycle;
        \draw (V23) -- (V2T);        
        \node[fill=white, inner sep=2pt] at (V21) {$\mathcal{BA}$};
        \node[fill=white, inner sep=2pt] at (V22) {$R(\mathcal{BA})$};
        \node[fill=white, inner sep=2pt] at (V23) {$R(Bip(\mathcal{BA}))$};
        \node[fill=white, inner sep=2pt] at (V24) {$Bip(\mathcal{BA})$};
        \node[fill=white, inner sep=2pt] at (V2T) {$B(\mathcal{BA})$};


    \end{tikzpicture}
    \caption{Some subvarieties of the variety $\mathcal{DMBL}$}
    \label{fig:lattice-regularisations}
\end{figure}

At least five questions naturally arise:
\begin{itemize}
    \item Can we find suitable generators for each of these varieties?
    \item Can we provide axiomatisations thereof?
    \item Can we characterise the De Morgan-P\l onka representations of their members in terms of their fibres and their underlying involutive semilattices?
    \item Are these varieties pairwise distinct?
    \item Together with the subvarieties of $\mathcal{DML}$, do they exhaust the lattice of subvarieties of $\mathcal{DMBL}$?
\end{itemize}
We will address all these questions and answer many of them.

\subsection{Types of regularisations: Characterisation results}

We start by characterising in several different manners the variety $\mathcal{DMBL} = B(\mathcal{DML})$ itself. The next theorem is, for the most part, a collection of results already available in the literature. The algebra $\mathbf{U}$, mentioned therein, is isomorphic to the De-Morgan P\l onka sum $\DPL(F)$ of distributive lattices over $\mathbf{IS}_4$ such that, if $\mathbf{D}_2$ and $\mathbf{D}_1$ are the $2$-element and the $1$-element distributive lattice, respectively, we have that $F(\mathbf{k}) = \mathbf{D}_1$, $F(\mathbf{j}) = F(\lnot \mathbf{j}) = \mathbf{D}_2$, and $F(\mathbf{i}) = \mathbf{D}_2 \times \mathbf{D}_2$ (see \cite{ferguson2015b}, \cite{PaoliSzmucZirattu}).

\begin{theorem}\label{accozzato}
    The following varieties are all coincident:
    \begin{enumerate}
        \item $\mathcal{DMBL}$;
        \item $\mathcal{DBL}^\ast$;
        \item $B(\mathcal{DML})$;
        \item $B(\mathcal{DL}^\ast)$;
        \item $(R(\mathcal{DL}))^\ast$;
        \item $V(\mathbf{DM}_4,\mathbf{IS}_4)$;
        \item $V(\mathbf{U})$;
        \item the class of De Morgan-P\l onka sums of distributive lattices.
    \end{enumerate}
\end{theorem}

\begin{proof}
The coincidence of (1) and (2), as well as the coincidence of (3) and (4), holds by definition. 

(2) coincides with (5) by the remark at the beginning of this section. 

(4) coincides with (8) by Theorem \ref{cappelloni}. 

(3) coincides with (6) by Theorem \ref{periscopio}.(2) and Fact \ref{fact:characterization-IS4}. 

(6) is included in (1), because both $\mathbf{DM}_4$ and $\mathbf{IS}_4$ are De Morgan bisemilattices. 

We show that (1) is included in (8). As shown in \cite{RandriamahazakaSL}, any De Morgan bisemilattice $\B$ is an a-involutive left-normal band with compatible operations, as witnessed by the operation $x \cdot y := x \land (x \lor y)$. By Theorem \ref{cappelloni}, then, $\B$ is representable as a De Morgan-P\l onka sum over $\langle B, \cdot, \lnot \rangle / \mathcal{D}$. However, each fibre of this representation satisfies, by construction, the identity $x \cdot y \approx x$, hence it is a distributive lattice.

To conclude our proof, it suffices to show that (6) is equivalent to (7). However, $\mathbf{DM}_4$ and $\mathbf{IS}_4$ are, respectively, a subalgebra and a quotient of $\mathbf{U}$. In turn, $\mathbf{U}$ is a subalgebra of the quotient of $\mathbf{DM}_4 \times \mathbf{IS}_4$ obtained by collapsing all and only the elements whose second coordinate is $\bf k$.
\end{proof}

Recall from Section \ref{inquina} that a \emph{regularised absorptive} De Morgan bisemilattice  \cite{PaoliSzmucZirattu} is a member of $\mathcal{DMBL}$ that satisfies the identity
\[
\text{R-Abs }x \land (x \lor y) \approx x \land (x \lor \lnot y).
\]
The variety of regularised absorptive De Morgan bisemilattices is noted $\mathcal{RDMBL}$. According to the next theorem, itself a collection of known results for the most part, $\mathcal{RDMBL} = R(\mathcal{DML})$. 

\begin{theorem}\label{servosterzo}
The following varieties are all coincident:
    \begin{enumerate}
        \item $\mathcal{RDMBL}$;
        \item $R(\mathcal{DML})$;
        \item $V(\mathbf{DM}_4,\mathbf{IS}_2)$
        \item $V(\mathbf{{DM}_4^\dagger})$;
        \item the class of De Morgan-P\l onka sums of distributive lattices over members of $V(\mathbf{IS}_2)$.
    \end{enumerate}
\end{theorem}

\begin{proof}
     The coincidence of (1) and (2) is proved in \cite{Hornischer} and the coincidence of the previous two items and (4) is established in \cite{PaoliSzmucZirattu}. Since $V(\mathbf{IS}_2)$ is term-equivalent to the variety of semilattices, the coincidence of (2) and (5) follows from Theorem \ref{cappelletti}. (3) is included in (1) because both $\mathbf{DM}_4$ and $\mathbf{IS}_2$ satisfy the axioms of $\mathcal{RDMBL}$. Finally, (4) is included in (3) because $\mathbf{{DM}_4^\dagger}$ is a quotient of $\mathbf{DM}_4 \times \mathbf{IS}_2$.
\end{proof}

We now proceed to proving a couple of results that are entirely new. A \emph{bipolarly absorptive} De Morgan bisemilattice  is a member of $\mathcal{DMBL}$ that satisfies the identity
\[
\text{B-Abs }x \land (x \lor \lnot x) \approx x \land (x \lor \lnot x) \land (x \lor \lnot x \lor y).
\] 
The variety of bipolarly absorptive De Morgan bisemilattices is noted $\mathcal{BDMBL}$.

\begin{theorem}\label{allagamento}
The following varieties are all coincident:
    \begin{enumerate}
        \item $\mathcal{BDMBL}$;
        \item $Bip(\mathcal{DML})$;
        \item $V(\mathbf{DM}_4,\mathbf{IS}_3)$;
        \item The class of De Morgan-P\l onka sums of distributive lattices over members of $V(\mathbf{IS}_3)$. 
    \end{enumerate}
\end{theorem}

\begin{proof}
    (1) is included in (4). Let $\A \in \mathcal{BDMBL}$. Since $\mathcal{BDMBL} \subseteq \mathcal{DMBL}$, by Theorem \ref{accozzato} $\A$ is representable as a De Morgan-P\l onka sum of distributive lattices. However, by (C-Abs) its a-involutive left-normal band term reduct $\langle A, \cdot, \lnot \rangle$ satisfies $x \cdot \lnot x \approx x \cdot \lnot x \cdot y$, whereby, by Lemma \ref{mezzitermini} $\langle A, \cdot, \lnot \rangle / \mathcal{D} \in \mathcal{BISL}$. Hence, by Proposition \ref{prop:characterization-IS3}, its involutive semilattice of indices belongs to $V(\mathbf{IS}_3)$.
    

    (2) coincides with (3). By definition, $Bip(\mathcal{DML})$ satisfies the identity $\varphi \approx \psi$ iff $\mathcal{DML}$ satisfies $\varphi \approx \psi$ and $\varphi \approx \psi$ is bipolarly balanced. By Theorem \ref{periscopio} and Proposition \ref{prop:characterization-IS3}, this happens iff both $\mathbf{DM}_4$ and $\mathbf{IS}_3$ satisfy $\varphi \approx \psi$.

    (3) is included in (1). It is easy to see that both $\mathbf{DM}_4$ and $\mathbf{IS}_3$ satisfy all the axioms of $\mathcal{BDMBL}$, whence the inclusion follows.

    (4) is included in (2). Let $\A$ be representable as a De Morgan-P\l onka sum of distributive lattices over some member of $V(\mathbf{IS}_3) = \mathcal{BISL}$. We need to establish that $\A$ satisfies all the bipolarly balanced identities satisfied in $\mathcal{DML}$. 
    
    We implement an adaptation of the proof found in \cite[Thm. 2.3.2]{BonzioPaoliPraBaldi}. Assume that $\varphi \approx \psi$ is valid in all De Morgan lattices and $\varphi \approx \psi$ is bipolarly balanced. From the fact that $\varphi \approx \psi$ is bipolarly balanced, we reason depending on whether it is a bipolar or a balanced identity. 
    
    If it is balanced, then, since it is also valid in all De Morgan lattices and hence in all bilateralisations of the fibres in the De Morgan-P\l onka representation of $\A$, by Theorem \ref{th:DM-Plonka sums and balanced equations} $\A \vDash \varphi \approx \psi$.
    
    If it is bipolar, then the conjunction of the facts that $Var^{+}(\varphi) \cap Var^{-}(\varphi) \neq \emptyset$ and $Var^{+}(\psi) \cap Var^{-}(\psi) \neq \emptyset$ and that the underlying involutive semilattice is in $\mathcal{BISL}$ entail that for all $h \in Hom(\mathbf{Fm}(\langle 2, 2, 1 \rangle), \A)$ we have that $h(\varphi)$ belongs to $A_{i\lor\lnot i \lor k}$, where $h(x) \in A_i$ and $h(\lnot x) \in A_{\lnot i}$ for some $x  \in Var^{+}(\varphi) \cap Var^{-}(\varphi)$ and $h(\psi)$ belongs to $A_{j\lor\lnot j \lor l}$, where $h(y) \in A_j$ and $h(\lnot y) \in A_{\lnot j}$ for some $y \in Var^{+}(\psi) \cap Var^{-}(\psi)$. However, since the underlying involutive semilattice is in $\mathcal{BISL}$, we know that $i \lor \lnot i \lor k \lor j \lor \lnot j \lor l = i \lor \lnot i$ and $j \lor \lnot j \lor l \lor i \lor \lnot i \lor k = j \lor \lnot j$, which ultimately entails that $i \lor \lnot i = j \lor \lnot j$. Whence, for all $h \in Hom(\mathbf{Fm}(\langle 2, 2, 1 \rangle), \A)$, $h(\varphi)$ and $ h(\psi)$ belong to the same fibre $F(i \lor \lnot i) = F(\lnot (i \lor \lnot i))$. Since $\varphi \approx \psi$ holds in all De Morgan lattices, hence in all bilateralisations of distributive lattices, and $i \lor \lnot i$ is a fixpoint of negation, it follows that $\A \vDash \varphi \approx \psi$.

\end{proof}

A \emph{regular bipolarly absorptive} De Morgan bisemilattice  is a member of $\mathcal{DMBL}$ that satisfies the identity
\[
\text{RB-Abs }x \land (x \lor \lnot x) \land (x \lor \lnot x \lor y) \approx x \land (x \lor \lnot x) \land (x \lor \lnot x \lor \lnot y).
\] 
The variety of regular bipolarly absorptive De Morgan bisemilattices is noted $\mathcal{RBDMBL}$.

\begin{theorem}\label{mareggiata}
The following varieties are all coincident:
    \begin{enumerate}
        \item $\mathcal{RBDMBL}$;
        \item $R(Bip(\mathcal{DML}))$;
        \item $V(\mathbf{DM}_4,\mathbf{IS}_2,\mathbf{IS}_3)$;
        \item The class of De Morgan-P\l onka sums of distributive lattices over members of $V(\mathbf{IS}_2,\mathbf{IS}_3)$. 
    \end{enumerate}
\end{theorem}

\begin{proof}

(1) is included in (4). Let $\A \in \mathcal{RBDMBL}$. Since $\mathcal{RBDMBL} \subseteq \mathcal{DMBL}$, by Theorem \ref{accozzato} $\A$ is representable as a De Morgan-P\l onka sum of distributive lattices. However, by (RB-Abs) its a-involutive left-normal band term reduct $\langle A, \cdot, \lnot \rangle$ satisfies $x \cdot \lnot x \cdot y \approx x \cdot \lnot x \cdot \lnot y$, whereby, by Lemma \ref{tiramolla} $\langle A, \cdot, \lnot \rangle / \mathcal{D} \in \mathcal{RBISL}$. Hence, by Corollary \ref{stritto}, its involutive semilattice of indices belongs to $V(\mathbf{IS}_2,\mathbf{IS}_3)$.
    

    (2) coincides with (3). By definition, $R(Bip(\mathcal{DML}))$ satisfies the identity $\varphi \approx \psi$ iff $\mathcal{DML}$ satisfies $\varphi \approx \psi$ and $\varphi \approx \psi$ is either bipolar and regular, or balanced regular. The latter conditions amount to being both regular and either bipolar or balanced. By Theorem \ref{periscopio} Fact  \ref{fact:characterization-IS2} and Proposition \ref{prop:characterization-IS3}, this happens iff $\mathbf{DM}_4$, $\mathbf{IS}_2$ and $\mathbf{IS}_3$ all satisfy $\varphi \approx \psi$.

    (3) is included in (1). It is easy to see that $\mathbf{DM}_4$, $\mathbf{IS}_2$ and $\mathbf{IS}_3$ satisfy all the axioms of $\mathcal{RBDMBL}$, whence the inclusion follows.

    (4) is included in (2). Let $\A$ be representable as a De Morgan-P\l onka sum of distributive lattices over some member of $V(\mathbf{IS}_2,\mathbf{IS}_3) = \mathcal{RBISL}$. We need to establish that $\A$ satisfies all the regular bipolarly balanced identities satisfied in $\mathcal{DML}$. 
    
    Assume that $\varphi \approx \psi$ is valid in all De Morgan lattices and $\varphi \approx \psi$ is regular bipolarly balanced. We distinguish cases according to whether they are bipolar and regular, or balanced regular. If the latter, we argue as in Theorem \ref{allagamento} and conclude that $\A \vDash \varphi \approx \psi$.
    
    If the former, then the conjunction of the facts that $Var(\varphi) = Var(\psi)$ and $Var^{+}(\varphi) \cap Var^{-}(\varphi) \neq \emptyset$ and $Var^{+}(\psi) \cap Var^{-}(\psi) \neq \emptyset$ and that the underlying involutive semilattice $\mathbf{I}$ is in $\mathcal{RBISL}$, entail that for all $h \in Hom(\mathbf{Fm}(\langle 2, 2, 1 \rangle), \A)$ we have that $h(\varphi)$ belongs to $A_{(i\lor\lnot i)\lor k}$, where $h(x) \in A_i$ and $h(\lnot x) \in A_{\lnot i}$ for some $x  \in Var^{+}(\varphi) \cap Var^{-}(\varphi)$ and $h(\psi)$ belongs to $A_{(j\lor\lnot j)\lor l}$, where $h(y) \in A_j$ and $h(\lnot y) \in A_{\lnot j}$ for some $y \in Var^{+}(\psi) \cap Var^{-}(\psi)$. Since $\varphi \approx \psi$ is a regular identity, we can assume that $i\lor\lnot i\lor k = \alpha^\mathbf{I} (v_1,...,v_n)$ and $j\lor\lnot j\lor l = \beta^\mathbf{I} (v_1,...,v_n)$, where $\alpha$ and $\beta$ are formulas of type $\langle 2, 1 \rangle$ actually containing the variables $v_1,...,v_n$. In particular, let $\alpha_1$ and $\alpha_2$ be the subformulas of $\alpha$ such that $\alpha_1^\mathbf{I} (v_1,...,v_n) = i, \alpha_2^\mathbf{I} (v_1,...,v_n) = k$, and let $\beta_1$ and $\beta_2$ be the subformulas of $\beta$ such that $\beta_1^\mathbf{I} (v_1,...,v_n) = j, \beta_2^\mathbf{I} (v_1,...,v_n) = l$.

    We have that:
    \begin{eqnarray*}
        i \lor \lnot i \lor k&=&\alpha_1^\mathbf{I} (\Vec{v}) \lor \lnot \alpha_1^\mathbf{I} (\Vec{v}) \lor \alpha_2^\mathbf{I} (\Vec{v}) \\
        &=&\alpha_1^\mathbf{I} (\Vec{v}) \lor \lnot \alpha_1^\mathbf{I} (\Vec{v}) \lor \alpha_2^\mathbf{I} (\Vec{v}) \lor \lnot \alpha_2^\mathbf{I} (\Vec{v}) \\
        &=& \beta_1^\mathbf{I} (\Vec{v}) \lor \lnot \beta_1^\mathbf{I} (\Vec{v}) \lor \beta_2^\mathbf{I} (\Vec{v}) \lor \lnot \beta_2^\mathbf{I} (\Vec{v})\\
        &=& \beta_1^\mathbf{I} (\Vec{v}) \lor \lnot \beta_1^\mathbf{I} (\Vec{v}) \lor \beta_2^\mathbf{I} (\Vec{v}) \\
        &=& j \lor \lnot j \lor l.
    \end{eqnarray*}
Here, the second and fourth equalities hold because $\mathbf{I} \in \mathcal{RBISL}$, while the third equality holds because of the behaviour of transition and dualising maps in De Morgan-P\l onka sums, given that $\alpha$ and $\beta$ contain the same variables.
    
As a consequence, $h(\varphi)$ and $h(\psi)$ belong to the same fibre $F(i \lor \lnot i \lor k ) = F(i \lor \lnot i \lor k \lor \lnot k) = F(\lnot (i \lor \lnot i \lor k \lor \lnot k))$. Since $\varphi \approx \psi$ holds in all De Morgan lattices, hence in all bilateralisations of distributive lattices, and $i \lor \lnot i \lor k \lor \lnot k$ is a fixpoint of negation, it follows that $\A \vDash \varphi \approx \psi$.
\end{proof}

\vspace{2mm}

As expected, the equivalences in the four theorems above are preserved---mutatis mutandis---by replacing $\mathcal{DML}$ by $\mathcal{KL}$ and $\mathcal{BA}$.

\begin{theorem}\label{cantuccio}
Let $\A \in \{ \mathbf{B}_2,\mathbf{K}_3 \} $. The following varieties are all coincident:

\begin{enumerate}
    \item $B(V(\A))$;
    \item $V(\A, {\bf IS}_4)$;
    \item the class $\mathcal{K}$ of isomorphic images of algebras in
    \[
    \{ \DPL(F) : F: \mathbf{I} \to \mathcal{DL}, \flat F(i) \in V(\A) \text{ for every }i \in I \}.
    \]
\end{enumerate}
\end{theorem}

\begin{proof}
The coincidence of (1) and (2) holds by Fact~\ref{fact:characterization-IS4}. 

We show the coincidence of (1) and (3). By definition, $B(V(\A))$ is the variety satisfying all and only the balanced regular identities valid in $V(\A)$. By Theorem \ref{th:DM-Plonka sums and balanced equations}, the class $\mathcal{K}$ satisfies all the balanced regular identities valid in all algebras of the form $\flat F(i)$, for $F$ having the indicated properties. Hence, it satisfies all the balanced regular identities valid in $V(\A)$. But since $\mathcal{K}$ contains $\mathbf{IS}_4$ and $\A$, it satisfies only balanced regular identities valid in $V(\A)$. Hence, our conclusion follows.
\end{proof}

\begin{theorem}
Let $\A \in \{ \mathbf{B}_2,\mathbf{K}_3 \} $. The following varieties are all coincident:

\begin{enumerate}
    \item $R(V(\A))$;
    \item $V(\A, {\bf IS}_2)$;
    \item $V(\A^\dagger)$;
    \item the class $\mathcal{K}$ of isomorphic images of algebras in
    \[
    \{ \DPL(F) : F: \mathbf{I} \to \mathcal{DL}, \mathbf{I} \in V(\mathbf{IS}_2),\flat F(i) \in V(\A) \text{ for every }i \in I \}.
    \]
\end{enumerate}
\end{theorem}

\begin{proof}All the equivalences follow from Theorem \ref{cappelletti} and other known results concerning regular varieties (e.g., \cite[p. 487]{LPR}) once we observe that 
\[
    \{ \DPL(F) : F: \mathbf{I} \to \mathcal{DL}, \mathbf{I} \in V(\mathbf{IS}_2),\flat F(i) \in V(\A) \text{ for every }i \in I \}.
    \]
is essentially the class of P\l onka sums of algebras in $V(\A)$.
\end{proof}

\begin{theorem}\label{pagnotta}
Let $\A \in \{ \mathbf{B}_2,\mathbf{K}_3 \} $. The following varieties are all coincident:

\begin{enumerate}
    \item $Bip(V(\A))$;
    \item $V(\A, {\bf IS}_3)$;
    \item the class $\mathcal{K}$ of isomorphic images of algebras in
    \[
    \{ \DPL(F) : F: \mathbf{I} \to \mathcal{DL}, \mathbf{I} \in V(\mathbf{IS}_3),\flat F(i) \in V(\A) \text{ for every }i \in I \}.
    \]
\end{enumerate}
\end{theorem}

\begin{proof}
The coincidence of (1) and (2) holds by Fact~\ref{prop:characterization-IS3}. 

The coincidence of (2) and (3) is proved as in Theorem \ref{cantuccio}, using Theorem \ref{allagamento} instead of Theorem \ref{th:DM-Plonka sums and balanced equations}.
\end{proof}

\begin{theorem}
Let $\A \in \{ \mathbf{B}_2,\mathbf{K}_3 \} $. The following varieties are all coincident:

\begin{enumerate}
    \item $R(Bip(V(\A)))$;
    \item $V(\A, {\bf IS}_2, {\bf IS}_3)$;
    \item the class $\mathcal{K}$ of isomorphic images of algebras in
    \[
    \{ \DPL(F) : F: \mathbf{I} \to \mathcal{DL}, \mathbf{I} \in V(\mathbf{IS}_2,\mathbf{IS}_3),\flat F(i) \in V(\A) \text{ for every }i \in I \}.
    \]
\end{enumerate}
\end{theorem}

\begin{proof}
The coincidence of (1) and (2) holds by Fact~\ref{fact:characterization-IS2} and Proposition~\ref{prop:characterization-IS3}. 

The coincidence of (2) and (3) is proved as in Theorem \ref{cantuccio}, using Theorem \ref{mareggiata} instead of Theorem \ref{th:DM-Plonka sums and balanced equations}.
\end{proof}

\subsection{Other subvarieties}

According to a well-known result by Dudek and Graczynska \cite{Dudek}, if $\mathcal{V}$ is a strongly irregular variety, the lattice $\mathbf{L}$ of subvarieties of $R(\mathcal{V})$ is isomorphic to $\mathbf{M} \times \mathbf{D}_2$, where $\mathbf{M}$ is the lattice of subvarieties of $\mathcal{V}$. In other words, the lattice of subvarieties of the regularisation of $\mathcal{V}$ is the direct product of the lattice of subvarieties of $\mathcal{V}$ and the lattice of subvarieties of the variety of semilattices.

This invites the conjecture that the lattice of subvarieties of the \emph{balanced} regularisation of a symmetric strongly irregular variety $\mathcal{V}$ is the direct product of the lattice of subvarieties of $\mathcal{V}$ and the lattice of subvarieties of the variety of \emph{involutive} semilattices. If true, this conjecture implies that the varieties in Figure \ref{fig:lattice-regularisations}, together with the varieties of involutive semilattices $\mathcal{T}, R(\mathcal{T}), Bip(\mathcal{T}), R(Bip(\mathcal{T}))$ and $B(\mathcal{T})$, are the only varieties in the lattice of subvarieties of $\mathcal{DMBL}$. 

It turns out that this conjecture is false in general, and in particular for $\mathcal{V} = \mathcal{DML}$, whence the lattice of subvarieties of $\mathcal{DMBL}$ does not admit this direct decomposition. Observe that, if $\mathcal{BA}$ is the variety of Boolean algebras, the Dudek-Graczynska result implies that the pair $\langle \mathcal{BA}, \mathcal{RISL}\rangle$ is a splitting pair in the lattice of subvarieties of $\mathcal{RDMBL} = R(\mathcal{DML})$. On the other hand, the pair $\langle \mathcal{BA}, \mathcal{ISL}\rangle$ is \emph{not} a splitting pair in the lattice of subvarieties of $\mathcal{DMBL} = B(\mathcal{DML})$, as the first item of the following Fact demonstrates. Indeed, let $\A_5 := \DPL(F)$, where $F: \mathbf{IS}_3 \to \mathcal{DL}$ is such that $F(\mathbf{i}) = F(\lnot \mathbf{i}) = \mathbf{D}_2$, while $F(\mathbf{j}) = \mathbf{D}_1$. For ease of reference, call $a$ (resp., $b)$ the top (resp. bottom) element of $F(\mathbf{i})$, and $u$ the unique element of $F(\mathbf{j})$.\footnote{$\A_5$ instantiates a construction which is related, but not identical, to the \emph{$0$-direct union} of semirings \cite{crvdol}. If we view $\mathbf{D}_2$ as a semiring, the $0$-direct union $I_0^\ast (\mathbf{D}_2)$ satisfies e.g. $\lnot (x \lor y) \approx \lnot x \lor \lnot y$ instead of the appropriate De Morgan identity.}

\begin{fact}
The following relations obtain between the varieties detailed below:
\begin{enumerate}
    \item $\mathcal{BA} \nsubseteq V(\A_5) \nsubseteq \mathcal{ISL}$.  
    \item $\mathcal{BISL} \subset V(\A_5) \subset Bip(\mathcal{BA})$.
\end{enumerate}
\end{fact}

\begin{proof}
    (1) Observe that the identity $x \land \lnot x \approx y \lor \lnot y$ holds in $\A_5$, while if fails in $\mathcal{BA}$. By the same token, the identity $x \land y \approx x \lor y$ holds in $\mathcal{ISL}$ and fails in $\A_5$, because $a \land b = b \neq a = a \lor b$.

    (2) Recall from Proposition \ref{prop:characterization-IS3} that $\mathcal{BISL}=V(\mathbf{IS_3})$, and observe that $\mathbf{IS_3}$ is a subalgebra (actually, a retract) of $\A_5$. By Theorem \ref{pagnotta}, $Bip(\mathcal{BA})$ is generated by $\B_2$ and $\mathbf{IS_3}$, and by the results in \cite{PaoliSzmucZirattu}, $\A_5$ itself is a quotient of $\B_2 \times \mathbf{IS_3}$. By the proof of Item (1), we have both that $\A_5 \notin \mathcal{BISL}$ and that $\A_5$ satisfies the identity $x \land \lnot x \approx y \lor \lnot y$, which fails in $Bip(\mathcal{BA})$. So the inclusions are proper.
\end{proof}

Given a strongly irregular variety $\mathcal{V}$ of dualised type $\mathcal{L}$, we let $Bip^-(\mathcal{V})$ be the variety that satisfies all and only the bipolarly balanced identities that are either valid in $\mathcal{V}$ or bipolar, $R(Bip^-(\mathcal{V}))$ be the variety that satisfies all and only the regular bipolarly balanced identities that are either valid in $\mathcal{V}$ or bipolar, and $B^-(\mathcal{V})$ be the variety that satisfies all and only the balanced regular identities that are either valid in $\mathcal{V}$ or bipolar. 

\begin{theorem}
The following relations obtain between the varieties detailed below:
\begin{enumerate}
        \item $Bip^-(\mathcal{DML}) = V(\A_5)$;
        \item $R(Bip^-(\mathcal{DML})) = V(\A_5, \mathbf{IS}_2)$;
        \item $B^-(\mathcal{DML}) = V(\A_5, \mathbf{IS}_4)$.
    \end{enumerate}
\end{theorem}

\begin{proof}
   (1) For the left-to-right inclusion, suppose that the identity $\varphi \approx \psi$ is  either (i) not bipolarly balanced, or (ii) invalid in $\mathbf{DM}_4$ and not bipolar. In the former case, by Proposition \ref{prop:characterization-IS3} it fails in $\mathbf{IS}_3$ and hence in $\A_5$, which contains $\mathbf{IS}_3$ as a subalgebra. In the latter, suppose w.l.g. that $Var^+(\varphi) \cap Var^-(\varphi) = \emptyset$. So there will be $h \in Hom(\mathbf{Fm}(\langle 2,2,1 \rangle), \A_5)$ such that $h(\varphi) \in \{a,b\}$. Hence, if it is possible for such an $h$ to be such that $h(\psi) \in \{\lnot a, \lnot b, u\}$, $\varphi \approx \psi$ fails in $\A_5$. If not, then for all $h \in Hom(\mathbf{Fm}(\langle 2,2,1 \rangle), \A_5)$ such that $h(\varphi) \in \{a,b\}$, we must have that $h(\psi) \in \{a,b\}$. But this entails that also $Var^+(\psi) \cap Var^-(\psi) = \emptyset$. So, essentially, $\varphi$ and $\psi$ are equivalent to identities in the language of lattices. Now, given that $\varphi \approx \psi$ fails in $\mathbf{DM}_4$ and that it can be viewed as an involution-free identity, then it must fail in distributive lattices. Whence, since $\mathbf{D}_2$ generates the variety of distributive lattices and is a $\langle \land,\lor \rangle$-subreduct of $\A_5$, we know that $\varphi \approx \psi$ fails in $\A_5$ too.

   For the converse inclusion, suppose that $\varphi \approx \psi$ is not valid in $\A_5$, i.e. there exists $h \in Hom(\mathbf{Fm}(\langle 2,2,1 \rangle), \A_5)$ such that $h(\varphi) \neq h(\psi)$. If either $h(\varphi),h(\psi) \in \{u,a,\lnot a\}$ or $h(\varphi),h(\psi) \in \{u,b,\lnot b\}$, then $\varphi \approx \psi$ is not valid in $\mathbf{IS}_3$ and hence, by Proposition \ref{prop:characterization-IS3}, it is not bipolarly balanced. Alternatively, were we to suppose that (w.l.g.) $h(\varphi)$ belongs to one of the preceding sets, while $h(\psi)$ does not, the following ensues. We could assume that $h(\varphi) = a, h(\psi) = b$ (noting that the other cases are similar). This means that $Var^+(\varphi) \cap Var^-(\varphi) = \emptyset$ and $Var^+(\psi) \cap Var^-(\psi) = \emptyset$. Whence, in particular, $\varphi \approx \psi$ is not bipolar. Arguing as in the previous direction, we can obtain a counterexample to its validity in distributive lattices and, thus, in $\mathbf{DM}_4$.

   (2) From left to right, suppose that the identity $\varphi \approx \psi$ is either (i) not regular bipolarly balanced, or (ii) invalid in $\mathbf{DM}_4$ and not bipolar. In case (ii), by the previous item, we know it fails in $\A_5$. In case (i), it is either not bipolarly balanced, or irregular. If the former, it fails in $\A_5$ by the previous item again. If the latter, it fails in $\mathbf{IS}_2$ by Fact \ref{fact:characterization-IS2}. Conversely, suppose that $\varphi \approx \psi$ has a counterexample either in $\A_5$ or in $\mathbf{IS}_2$. If the former, then by the previous item it is either not bipolarly balanced (thus, in particular, not regular bipolarly balanced), or both invalid in $\mathbf{DM}_4$ and not bipolar. If the latter, then by Fact \ref{fact:characterization-IS2} it is irregular, hence not regular bipolarly balanced.

   (3) Similar, using Fact \ref{fact:characterization-IS4}.
\end{proof}

Clearly, the preceding proof does not depend on any particular property of De Morgan lattices, other than that they satisfy exactly the negation-free identities that hold in distributive lattices. The same, however, holds for Kleene lattices and Boolean algebras. Therefore, we have that:

\begin{corollary}
The following relations obtain between the varieties detailed below:
\begin{enumerate}
        \item $Bip^-(\mathcal{DML}) = Bip^-(\mathcal{KL}) = Bip^-(\mathcal{BA})$.
        \item $R(Bip^-(\mathcal{DML})) = R(Bip^-(\mathcal{KL})) = R(Bip^-(\mathcal{BA}))$.
        \item $B^-(\mathcal{DML}) = B^-(\mathcal{KL}) = B^-(\mathcal{BA})$.
    \end{enumerate}
\end{corollary}

Also, observe that none of the above varieties is an extension of $\mathcal{BA}$, because they all satisfy $x \land \lnot x \approx x \lor \lnot x$.

\section{The structure of the lattice of subvarieties}\label{fireworks}

In this final section, we provide a complete description of the lattice of subvarieties of $\mathcal{DMBL}$. Recall that, by \cite[Lm. 4.4]{PaoliSzmucZirattu}, if $\A$ is a subdirectly irreducible De Morgan bisemilattice, its De Morgan-P\l onka representation has a subdirectly irreducible involutive semilattice of indices. Moreover, by \cite[Thm. 3.3]{PaoliSzmucZirattu} the algebra $\mathbf{U}$ from Theorem \ref{accozzato} generates $\mathcal{DMBL}$ as a quasivariety. Also, recall the following \textquotedblleft relativised'' version of J\'onsson's Lemma (see \cite{czeldziob}):

\begin{theorem}
    Let $\mathcal{K}$ be a finite class of finite algebras, and let $ISP(\mathcal{K})$ be a variety. Then all subdirectly irreducible members of $ISP(\mathcal{K})$ are in $IS(\mathcal{K})$.
\end{theorem}

As a consequence, all the subdirectly irreducible De Morgan bisemilattices must be subalgebras of $\mathbf{U}$, and there are only finitely many such algebras. In principle, however, every subset of this set could generate a distinct subvariety of $\mathcal{DMBL}$. We will see that the actual distinct subvarieties are way fewer, though. Let us see how we can prune these potential sets of generators.

First of all, the next theorem allows us to \textquotedblleft weed out\textquotedblright ~some subdirectly irreducible algebras $\A \simeq \DPL(F)$ such that $F$ has domain $\mathbf{IS}_4$, as long as we are only interested in determining the structure of our subvariety lattice.

\begin{theorem}\label{cafecito}
    Let $\A \simeq \DPL (F)$ be a subdirectly irreducible De Morgan bisemilattice, where $F$ has domain $\mathbf{IS}_4$. Then $V(\A) = V(F(\mathbf{i})^\lnot \times \mathbf{IS}_4) = V(\B,\mathbf{IS}_4)$, where $\B \in \{\mathbf{DM}_4,\mathbf{K}_3, \mathbf{B}_2, \mathbf{IS}_1 \}$.
\end{theorem}

\begin{proof}
    For a start, we show the first equality. By \cite[Thm. 4.1]{PaoliSzmucZirattu}, if $\A \simeq \DPL (F)$ is subdirectly irreducible and $F$ has domain $\mathbf{IS}_4$, then two cases may arise: either $\A = \mathbf{IS}_4$, or $F(\mathbf{k}) = \mathbf{D}_1$ and $F(\mathbf{j}) = F(\lnot \mathbf{j}) = \mathbf{D}_2$. On the other hand, it is readily seen that $F(\mathbf{i})^\lnot \times \mathbf{IS}_4$ is isomorphic to an algebra of the form $\DPL (G)$, where $G: \mathbf{IS}_4 \to \mathcal{DL}$ is such that $G(\mathbf{i}) = G(\mathbf{j}) = G(\mathbf{k}) = F(\mathbf{i}), G(\lnot \mathbf{j}) = F(\mathbf{i})^d$. Now, $G(\mathbf{k})$ is a \emph{sink} of $\DPL (G)$, whence, by \cite[Lm. 4.5.15]{RS}, the equivalence $\theta$ collapsing all and only members of $G(\mathbf{k})$ is a congruence on $\DPL (G)$. However, $\DPL (F)$ is a subalgebra of $\DPL (G)/\theta$. Consequently, in either of the above cases, $\A$ is an isomorphic image of a subalgebra of a quotient of an isomorphic image of $F(\mathbf{i})^\lnot \times \mathbf{IS}_4$. For the other direction, $F(\mathbf{i})^\lnot$ and $\mathbf{IS}_4$ are, respectively, a subalgebra and a quotient of $\DPL (F) \simeq \A$. Thus, the first equality is established.

    For the second equality, as observed in Subsection \ref{inquina}, $F(\mathbf{i})^\lnot$ is a De Morgan lattice, hence, by Theorem \ref{periscopio}, it is a subdirect product whose factors are in $\{\mathbf{DM}_4,\mathbf{K}_3, \mathbf{B}_2, \mathbf{IS}_1 \}$. If $F(\mathbf{i})^\lnot$ is trivial, the equality is clearly true. If it is not, we distinguish three cases: (i) $F(\mathbf{i})^\lnot \in \mathcal{BA}$, (ii) $F(\mathbf{i})^\lnot \in \mathcal{KL} \setminus \mathcal{BA}$, (ii) $F(\mathbf{i})^\lnot \in \mathcal{DML} \setminus \mathcal{KL}$. In case (ii) (the other cases are similar), the factors in the subdirect representation of $F(\mathbf{i})^\lnot$ are in $\{\mathbf{K}_3, \mathbf{B}_2, \mathbf{IS}_1 \}$. So $V(\mathbf{K}_3,\mathbf{IS}_4)$ will contain $F(\mathbf{i})^\lnot$. But $V(F(\mathbf{i})^\lnot,\mathbf{IS}_4)$ will contain $\mathbf{K}_3$, too, because $F(\mathbf{i})^\lnot \in \mathcal{KL} \setminus \mathcal{BA}$. So, our equality follows.
\end{proof}

If $\A \in \mathcal{K}$ has $\mathbf{IS}_4$ as an involutive semilattice of indices, Theorem \ref{cafecito} warrants its replacement with a subdirectly irreducible De Morgan lattice together with $\mathbf{IS}_4$. Also, by \cite[Thm. 4.1]{PaoliSzmucZirattu} there are only $10$ subdirectly irreducible De Morgan bisemilattices over $\mathbf{IS}_1,\mathbf{IS}_2,\mathbf{IS}_3$. Adding $\mathbf{IS}_4$ to them, we obtain an $11$-element set 
    \[
    \mathcal{S} = \{ \A: \A \text{ is a s.i. De Morgan-P\l onka sum over } \mathbf{IS}_1,\mathbf{IS}_2,\mathbf{IS}_3 \} \cup \{\mathbf{IS}_4\},
    \] 
so that there can be at most as many distinct subvarieties of $\mathcal{DML}$ as there are subsets of $\mathcal{S}$. The $11$ De Morgan bisemilattices in $\mathcal{S}$ are depicted in Figure \ref{fig:subdirectly}; ovals denote the fibres of their De Morgan-P\l onka representations, while in each case the transition and dualising maps are uniquely determined.

\newsavebox{\boxI}
\savebox{\boxI}{
\begin{tikzpicture}[scale=.25]
  \node (zero) at (0,-2) {$\bullet$};
\end{tikzpicture}
}

\newsavebox{\boxII}
\savebox{\boxII}{
\begin{tikzpicture}[scale=.3]
        \node (a) at (0,0) {$\bullet$};
         \node (z) at (0,-4) {$\bullet$};       
        \draw (a) -- (z);
\end{tikzpicture}
}

\newsavebox{\boxIII}
\savebox{\boxIII}{
\begin{tikzpicture}[scale=.4]
        \node (a) at (0,0) {$\bullet$};
        \node (zero) at (0,-2) {$\bullet$};
         \node (z) at (0,-4) {$\bullet$};       
        \draw (a) -- (zero) -- (z);
\end{tikzpicture}
}

\newsavebox{\boxIV}
\savebox{\boxIV}{
\begin{tikzpicture}[scale=.35]
 \node (one) at (0,-4) {$\bullet$}; 
    \node (a) at (-2,-2) {$\bullet$}; 
    \node (b) at (2,-2) {$\bullet$}; 
    \node (zero) at (0,0) {$\bullet$}; 
    \draw (zero) -- (a) -- (one) -- (b) -- (zero);
\end{tikzpicture}
}

\newsavebox{\boxV}
\savebox{\boxV}{
\begin{tikzpicture}[scale=.35]
        \node [draw, solid, circle](a) at (0,0) {$\bullet$};
         \node[draw, solid, circle] (z) at (0,-6) {$\bullet$};       
        \draw (a) -- (z);
\end{tikzpicture} 
}

\newsavebox{\boxVI}
\savebox{\boxVI}{
\begin{tikzpicture}[scale=.35]
        \node [draw, solid, circle](a) at (0,0) {$\bullet$};
         \node[draw, solid, ellipse] (z) at (0,-8) {\usebox{\boxII}};       
        \draw (a) -- (z);
\end{tikzpicture} 
}

\newsavebox{\boxVII}
\savebox{\boxVII}{
\begin{tikzpicture}[scale=.35]
        \node [draw, solid, circle](a) at (0,0) {$\bullet$};
         \node[draw, solid, ellipse] (z) at (0,-8) {\usebox{\boxIII}};       
        \draw (a) -- (z);
\end{tikzpicture} 
}

\newsavebox{\boxVIII}
\savebox{\boxVIII}{
\begin{tikzpicture}[scale=.35]
        \node [draw, solid, circle](a) at (0,0) {$\bullet$};
         \node[draw, solid, circle] (z) at (0,-8) {\usebox{\boxIV}};       
        \draw (a) -- (z);
\end{tikzpicture} 
}

\newsavebox{\boxIX}
\savebox{\boxIX}{
\begin{tikzpicture}[scale=.35]
  \node [draw, solid, circle] (one) at (0,2) {$\bullet$};
  \node [draw, solid, circle] (a) at (-2,0) {$\bullet$};
  \node [draw, solid, circle] (b) at (2,0) {$\bullet$};
  \draw (a) -- (one) -- (b) ;
\end{tikzpicture} 
}

\newsavebox{\boxX}
\savebox{\boxX}{
\begin{tikzpicture}[scale=.35]
  \node [draw, solid, circle] (one) at (0,2) {$\bullet$};
  \node [draw, solid, ellipse] (a) at (-4,-4) {\usebox{\boxII}};
  \node [draw, solid, ellipse] (b) at (4,-4) {\usebox{\boxII}};
  \draw (a) -- (one) -- (b) ;
\end{tikzpicture} 
}

\newsavebox{\boxXI}
\savebox{\boxXI}{
\begin{tikzpicture}[scale=.35]
  \node [draw, solid, circle] (one) at (0,2) {$\bullet$};
  \node [draw, solid, circle] (a) at (-2,0) {$\bullet$};
  \node [draw, solid, circle] (b) at (2,0) {$\bullet$};
  \node [draw, solid, circle] (zero) at (0,-2) {$\bullet$};
  \draw (zero) -- (a) -- (one) -- (b) -- (zero);
\end{tikzpicture} 
}

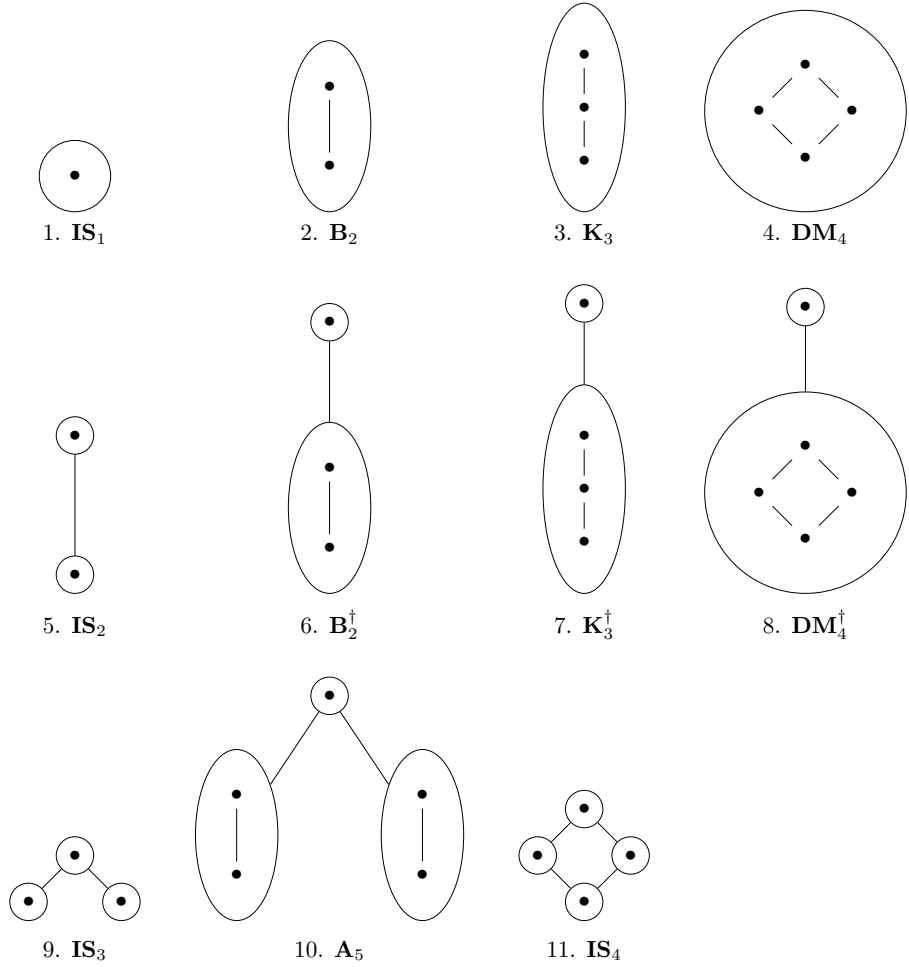
\begin{figure}[h]
        \centering

        \begin{adjustbox}{width=\linewidth, keepaspectratio}
\begin{tabular}{cccc}
\begin{tikzpicture}
  \node[draw, solid, circle] at (0,0) {\usebox{\boxI}};
\end{tikzpicture}
&
\begin{tikzpicture}
  \node[draw, solid, ellipse]  at (0,0) {\usebox{\boxII}};
\end{tikzpicture}
&
\begin{tikzpicture}
  \node[draw, solid, ellipse] at (0,0)  {\usebox{\boxIII}};
\end{tikzpicture}
&
\begin{tikzpicture}
  \node[draw, solid, circle]  at (0,0) {\usebox{\boxIV}};
\end{tikzpicture}
\\
1. ${\bf IS}_1$ & 2. ${\bf B}_2$ & 3. ${\bf K}_3$ & 4. ${\bf DM}_4$\\
 & & & \\
\begin{tikzpicture}
  \node  at (0,0) {\usebox{\boxV}};
\end{tikzpicture}
&
\begin{tikzpicture}
  \node  at (0,0) {\usebox{\boxVI}};
\end{tikzpicture}
&
\begin{tikzpicture}
  \node  at (0,0) {\usebox{\boxVII}};
\end{tikzpicture}
&
\begin{tikzpicture}
  \node  at (0,0) {\usebox{\boxVIII}};
\end{tikzpicture}
\\
5. ${\bf IS}_2$ & 6. ${\bf B}_2^{\dagger}$ & 7. ${\bf K}_3^{\dagger}$ & 8. ${\bf DM}_4^{\dagger}$\\
 & & & \\
\begin{tikzpicture}
  \node  at (0,0) {\usebox{\boxIX}};
\end{tikzpicture}
&
\begin{tikzpicture}
  \node  at (0,0) {\usebox{\boxX}};
\end{tikzpicture}
&
\begin{tikzpicture}
  \node  at (0,0) {\usebox{\boxXI}};
\end{tikzpicture}
&
\\
9. ${\bf IS}_3$ 
&
10. ${\bf A}_5$
&
11. ${\bf IS}_4$
&
\\
\end{tabular}
\end{adjustbox}

        \caption{Subdirectly irreducible De Morgan bisemilattices in $\mathcal{S}$}
        \label{fig:subdirectly}
    \end{figure}
The standard technique for carrying out this task consists in: (i) computing the preorder

\[
\preceq := \{\langle \A,\B \rangle: \A,\B \in \mathcal{S}, \A \in HSP(\B)\}
\]

and the induced equivalence relation ${\equiv}$ defined as ${\preceq \cap \preceq^{-1}}$, and (ii) determining the order-ideals of the poset $\langle \mathcal{S}/\equiv, \subseteq \rangle$. For an $11$-element set, (i) is computationally heavy. However, since $\mathcal{RDMBL} = R(\mathcal{DML})$ is the regularisation of a strongly irregular variety, its lattice of subvarieties is known to contain precisely the $4$ subvarieties of $\mathcal{DML}$ and their regularisations, all of them pairwise distinct. Therefore, it will be enough to describe the subvarieties of $\mathcal{DMBL}$ that are \emph{not} below $\mathcal{RDMBL}$. We will do this \textquotedblleft manually\textquotedblright, by enforcing some constraints that suffice to prove that the lattice of subvarieties of $\mathcal{DMBL}$ contains \emph{at most} the varieties of the form $\mathcal{V},B(\mathcal{V}),R(Bip(\mathcal{V})),Bip(\mathcal{V}),R(\mathcal{V})$ and $B^-(\mathcal{V}),R(Bip^-(\mathcal{V})),Bip^-(\mathcal{V})$, for $\mathcal{V} \in  \{ \mathcal{T}, \mathcal{BA}, \mathcal{KL}, \mathcal{DML}\}$, and then proving that these varieties are pairwise distinct.

\begin{theorem}\label{principale}
The following hold of $\mathcal{DMBL}$:
    \begin{enumerate}
        \item There are at most $23$ subvarieties of $\mathcal{DMBL}$.
        \item The lattice of subvarieties of $\mathcal{DMBL}$ is the lattice depicted in Figure \ref{fig:subvarieties}: hence, in particular, there are precisely $23$ subvarieties of $\mathcal{DMBL}$.
    \end{enumerate}
\end{theorem}

\begin{figure}[h]
    \centering

    \resizebox{\linewidth}{!}{
    \begin{tikzpicture}[font=\tiny, inner sep=0pt] 
        
        \def\squish{0.7}

        \coordinate (V01) at (0,0);
        \coordinate (V02) at (\squish, 1);  
        \coordinate (V03) at (0, 2);
        \coordinate (V0T) at (0, 3);        
        \coordinate (V04) at (-\squish, 1); 

        \coordinate (V11) at (3.5, -1);
        \coordinate (V12) at (3.5+\squish, 0);
        \coordinate (V13) at (3.5, 1);
        \coordinate (V1T) at (3.5, 2);        
        \coordinate (V14) at (3.5-\squish, 0);

        \coordinate (V21) at (7, -2);
        \coordinate (V22) at (7+\squish, -1);
        \coordinate (V23) at (7, 0);
        \coordinate (V2T) at (7, 1);        
        \coordinate (V24) at (7-\squish, -1);

        \coordinate (V31) at (10.5, -3);
        \coordinate (V32) at (10.5+\squish, -2);
        \coordinate (V33) at (10.5, -1);
        \coordinate (V3T) at (10.5, 0);        
        \coordinate (V34) at (10.5-\squish, -2);

        \coordinate (VE1) at (4, -3.5);        
        \coordinate (VE2) at (4, -2.25);
        \coordinate (VE3) at (4, -1.5);

        \draw (VE1) -- (V24);
        \draw (VE2) -- (V23);
        \draw (VE3) -- (V2T);
        \draw (V34) -- (VE1) -- (VE2) -- (VE3);

        \draw (V01) -- (V11) -- (V21) -- (V31);
        \draw (V02) -- (V12) -- (V22) -- (V32);
        \draw (V03) -- (V13) -- (V23) -- (V33);
        \draw (V04) -- (V14) -- (V24) -- (V34);
        \draw (V0T) -- (V1T) -- (V2T) -- (V3T);

        \foreach \i in {0,1,2,3} {
            \draw (V\i1) -- (V\i2) -- (V\i3) -- (V\i4) -- cycle;
            \draw (V\i3) -- (V\i T);
        }

        \node[fill=white, inner sep=2pt] at (V01) {$\mathcal{DML}$};
        \node[fill=white, inner sep=2pt] at (V02) {$R(\mathcal{DML})$};
        \node[fill=white, inner sep=2pt] at (V03) {$R(Bip(\mathcal{DML}))$};
        \node[fill=white, inner sep=2pt] at (V04) {$Bip(\mathcal{DML})$};
        \node[fill=white, inner sep=2pt] at (V0T) {$B(\mathcal{DML})$};        

        \node[fill=white, inner sep=2pt] at (V11) {$\mathcal{KL}$};
        \node[fill=white, inner sep=2pt] at (V12) {$R(\mathcal{KL})$};
        \node[fill=white, inner sep=2pt] at (V13) {$R(Bip(\mathcal{KL}))$};
        \node[fill=white, inner sep=2pt] at (V14) {$Bip(\mathcal{KL})$};
        \node[fill=white, inner sep=2pt] at (V1T) {$B(\mathcal{KL})$};      

        \node[fill=white, inner sep=2pt] at (V21) {$\mathcal{BA}$};
        \node[fill=white, inner sep=2pt] at (V22) {$R(\mathcal{BA})$};
        \node[fill=white, inner sep=2pt] at (V23) {$R(Bip(\mathcal{BA}))$};
        \node[fill=white, inner sep=2pt] at (V24) {$Bip(\mathcal{BA})$};
        \node[fill=white, inner sep=2pt] at (V2T) {$B(\mathcal{BA})$};

        \node[fill=white, inner sep=2pt] at (V31) {$\mathcal{T}$};
        \node[fill=white, inner sep=2pt] at (V32) {$R(\mathcal{T})$};
        \node[fill=white, inner sep=2pt] at (V33) {$R(Bip(\mathcal{T}))$};
        \node[fill=white, inner sep=2pt] at (V34) {$Bip(\mathcal{T})$};
        \node[fill=white, inner sep=2pt] at (V3T) {$B(\mathcal{T})$};

        \node[fill=white, inner sep=2pt] at (VE1) {$Bip^{-}(\mathcal{DML})$};
        \node[fill=white, inner sep=2pt] at (VE2) {$R(Bip^{-}(\mathcal{DML}))$};
        \node[fill=white, inner sep=2pt] at (VE3) {$B^{-}(\mathcal{DML})$};
        
    \end{tikzpicture}
    }    
        
        \caption{All varieties of De Morgan bisemilattices}
    \label{fig:subvarieties}
\end{figure}
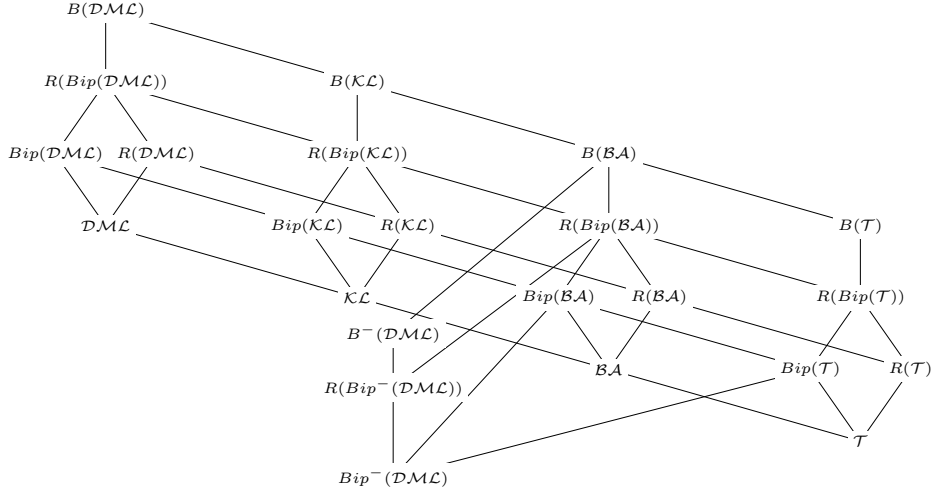
\begin{proof}
    (1) If we run the Python code in Figure \ref{fig:placeholder}, we obtain $15$ valid subsets of subdirectly irreducible algebras that respect the enforced constraints, giving rise to the Hasse diagram in Figure \ref{fig:subvarieties}, where the edge between $Bip(\mathcal{T})$ and $Bip^{-}(\mathcal{DML})$, thus represented for typographical reasons, should be understood as pointing upwards. We now proceed to justify these constraints.
    
\begin{figure}[h]
    \centering
    \begin{lstlisting}
from itertools import chain, combinations
A = set(range(1, 12))
def powerset(s):
    s = list(s)
    return chain.from_iterable(combinations(s, r) for r in range(len(s) + 1))
def valid(S):
    S = set(S)
    # (a) must contain 1
    if 1 not in S:
        return False
    # (b) must contain at least one number >= 9
    if not any(x >= 9 for x in S):
        return False
    # (c) if 11, then 9 and 5
    if 11 in S and not ({9, 5} <= S):
        return False
    # (d) if 10, then 9
    if 10 in S and 9 not in S:
        return False
    # (e) if 8, then 7,6,5,4,3,2
    if 8 in S and not ({7,6,5,4,3,2} <= S):
        return False
    # (f) if 7, then 6,5,3,2
    if 7 in S and not ({6,5,3,2} <= S):
        return False
    # (g) if 6, then 5 and 2
    if 6 in S and not ({5,2} <= S):
        return False
    # (h) if 4, then 3 and 2
    if 4 in S and not ({3,2} <= S):
        return False
    # (i) if 3, then 2
    if 3 in S and 2 not in S:
        return False
    # (j) 10 and 2 iff 9 and 2
    if ((10 in S and 2 in S) != (9 in S and 2 in S)):
        return False
    # (k) 10 and 3 iff 9 and 3
    if ((10 in S and 3 in S) != (9 in S and 3 in S)):
        return False
    # (l) 10 and 4 iff 9 and 4
    if ((10 in S and 4 in S) != (9 in S and 4 in S)):
        return False
    # (m) 6 iff 2 and 5
    if ((6 in S) != ({2,5} <= S)):
        return False
    # (n) 7 iff 3 and 5
    if ((7 in S) != ({3,5} <= S)):
        return False
    # (o) 8 iff 4 and 5
    if ((8 in S) != ({4,5} <= S)):
        return False
    return True
# Generate and sort all valid subsets
valid_sets = sorted(
    [sorted(S) for S in powerset(A) if valid(S)],
    key=lambda x: (len(x), x))
print("Number of valid subsets:", len(valid_sets))
print("\nValid subsets:")
for s in valid_sets:
    print(s)
\end{lstlisting}
    \caption{Constraints on the generating sets for subvarieties of $\mathcal{DMBL}$}
    \label{fig:placeholder}
\end{figure}

    Constraint (a) is clear. Constraint (b) is there because we are determining the subvarieties of $\mathcal{DMBL}$ that are not below $\mathcal{RDMBL}$. Constraint (c) is in place because $\mathbf{IS}_2$ and $\mathbf{IS}_3$ are subalgebras of $\mathbf{IS}_4$. Constraint (d) holds because $\mathbf{IS}_3$ is a subalgebra of $\A_5$. Constraints (h) and (i) hold because $\B_2$ is a subalgebra of $\mathbf{K}_3$, which in turn is a subalgebra of $\mathbf{DM}_4$. Constraint (g) holds because $\B_2$ is a subalgebra of $\B_2^\dagger$ and $\mathbf{IS}_2$ is a quotient of it; similarly for Constraints (e) and (f).

    As regards Constraints (m), (n), (o), their left-to-right directions follow from Constraints (e), (f), (g). For the other direction, observe that, for $\A \in \{ \B_2, \mathbf{K}_3, \mathbf{DM}_4 \}$, $\A^\dagger$ is a quotient of $\A \times \mathbf{IS}_2$. 

    Finally, we justify Constraints (j), (k), (l) by showing that for $\A \in \{ \B_2, \mathbf{K}_3, \mathbf{DM}_4 \}$, the $V(\A_5, \A)$ coincides with $V(\mathbf{IS}_3, \A)$, i.e., with $Bip(V(\A))$. One direction follows from the fact that $\mathbf{IS}_3$ is a subalgebra of $\A_5$---see Constraint (d). For the converse inclusion, observe that $\A_5$ is a subalgebra of $\A \times \mathbf{IS}_3$.

    (2) The inclusions in Figure \ref{fig:subvarieties} follow from (i) and from the Theorems in Section \ref{morgana}. It remains to be shown that these varieties are pairwise distinct.

    It is well-known that $\mathcal{T} \subset \mathcal{BA} \subset \mathcal{KL} \subset \mathcal{DML}$. It follows from the results in \cite{Dolinka2000} that the displayed subvarieties of $B(\mathcal{T})$ are pairwise distinct (and distinct from all varieties extending $\mathcal{BA}$), while it follows from the results in \cite{Dudek} that the operator $R$ is injective and that $\mathcal{V} \subset R(\mathcal{V})$, for $\mathcal{V} \in \{\mathcal{T}, \mathcal{BA}, \mathcal{KL}, \mathcal{DML}\}$.

    The next table shows that the operators $Bip(.),R(Bip(.)),B(.)$ are injective by presenting, for varieties $\mathcal{V},\mathcal{W}$ such that $\mathcal{V} \subset \mathcal{W}$, an identity that holds in $Bip(\mathcal{V})$ (respectively, in $R(Bip(\mathcal{V})), B(\mathcal{V})$) but not in $Bip(\mathcal{W})$ (respectively, in $R(Bip(\mathcal{W})), B(\mathcal{W})$). We use the abbreviations $\overline{x}$ for $x \lor \lnot x$ and $\underline{x}$ for $x \land \lnot x$:
    \begin{eqnarray*}
Bip(\mathcal{BA}) \subset Bip(\mathcal{KL}) & \underline{x} \approx \underline{y}\\
Bip(\mathcal{KL}) \subset Bip(\mathcal{DML})& \underline{x} \approx \underline{x} \land \overline{y} \\
R(Bip(\mathcal{BA})) \subset R(Bip(\mathcal{KL})) & \underline{x} \land y \approx \underline{x} \land \lnot y\\
R(Bip(\mathcal{KL})) \subset R(Bip(\mathcal{DML})) & \overline{x} \land \overline{y} \land (\underline{x} \lor \underline{y}) \approx \underline{x} \lor \underline{y}\\
B(\mathcal{BA}) \subset B(\mathcal{KL}) & \underline{x} \land ( \underline{x} \lor \underline{y}) \approx \underline{y} \land ( \underline{y} \lor \underline{x})\\
B(\mathcal{KL}) \subset B(\mathcal{DML}) & \overline{x} \land \overline{y} \land (\underline{x} \lor \underline{y}) \approx \underline{x} \lor \underline{y}\\
    \end{eqnarray*}
For $\mathcal{V},\mathcal{W} \in \{\mathcal{T}, \mathcal{BA}, \mathcal{KL}, \mathcal{DML}\}$, we have that $R(\mathcal{V}) \neq Bip(\mathcal{W}),R(\mathcal{V}) \neq R(Bip(\mathcal{W}))$ and $R(\mathcal{V}) \neq B(\mathcal{W})$ because $R(\mathcal{V})$ satisfies $x \land (x \lor y) \approx x \land (x \lor \lnot y)$ while $Bip(\mathcal{W}), R(Bip(\mathcal{W})), B(\mathcal{W})$ do not. We also have that $Bip(\mathcal{V}) \neq R(Bip(\mathcal{W}))$ because $Bip(\mathcal{V})$ satisfies $x \land (x \lor y \lor \lnot y) \approx x \land (x \lor \lnot x) $ while $R(Bip(\mathcal{W}))$ does not. Finally, we have that $Bip(\mathcal{V}) \neq B(\mathcal{W})$ and $R(Bip(\mathcal{V})) \neq B(\mathcal{W})$ because $Bip(\mathcal{V}),R(Bip(\mathcal{V}))$ satisfy $x \land (x \lor y \lor \lnot y) \approx x \land (x \lor \lnot x \lor y) $, while $B(\mathcal{W})$ does not.

To conclude, let us deal with $Bip^-(\mathcal{DML}),R(Bip^-(\mathcal{DML})),B^-(\mathcal{DML})$. They satisfy the inclusions $Bip^-(\mathcal{DML}) \subset R(Bip^-(\mathcal{DML})) \subset B^-(\mathcal{DML})$, as witnessed from the fact that $Bip^-(\mathcal{DML})$ satisfies $x \land \lnot x \approx y \lor \lnot y$ while $R(Bip^-(\mathcal{DML}))$ does not, and the fact that $R(Bip^-(\mathcal{DML}))$ satisfies $x \land \lnot x \land y \approx x \land \lnot x \land \lnot y$ while $B^-(\mathcal{DML})$ does not. These three varieties are also distinct from all involutive semilattice varieties, because they do not satisfy $x \land y \approx x \lor y$, and from all extensions of $\mathcal{BA}$, because they satisfy $x \land \lnot x \approx x \lor \lnot x$ while any extension of $\mathcal{BA}$ fails it.
\end{proof}

\emph{Acknowledgements.} We gratefully acknowledge the support of the following funding sources:
\begin{itemize}
    \item the European Union (Horizon Europe Research and Innovation Programme), through the MSCA-RISE action PLEXUS (Grant Agreement no 101086295);
    \item the Italian Ministry of Education, University and Research, through the PRIN 2022 project DeKLA (``Developing Kleene Logics and their Applications'', project code: 2022SM4XC8) and the PRIN Pnrr project ``Quantum Models for Logic, Computation and Natural Processes (Qm4Np)'' (project code: P2022A52CR);
    \item Fondazione di Sardegna, through the project ``An algebraic approach to hyperintensionality'', project code: F23C25000360007;
    \item the French Ministère de l’Europe et des Affaires Étrangères (MEAE) and Ministère de l’Enseignement Supérieur, de la Recherche et de l’Innovation (MESRI), and the Argentinian Ministerio de Ciencia, Tecnología e Innovación (MinCyT) and Consejo Nacional de Investigaciones Científicas y Técnicas (CONICET), through the bilateral ECOS Sud project ``Logical consequence and many-valued models'' (Grant A22H01)
    \item the Northwestern Italian Philosophy Consortium (FINO) and the Fondazione Franco e Marilisa Caligara (Turin, Italy).
\end{itemize}

\bibliography{references}
\bibliographystyle{abbrv}

\end{document}